\documentclass[11pt]{amsart}

\usepackage[utf8]{inputenc}
\usepackage[colorlinks]{hyperref}
\usepackage{mathtools, amsthm, amsfonts, amssymb, amsmath}
\usepackage{accents} 
\usepackage{microtype}
\usepackage{multicol}
\usepackage{cleveref} 

\hypersetup{
  linkcolor=[rgb]{0.3,0.3,0.6},
  citecolor=[rgb]{0.2, 0.6, 0.2},
  urlcolor=[rgb]{0.6, 0.2, 0.2}
}

\usepackage[colorlinks]{hyperref}
\makeindex

\usepackage{thmtools}
\usepackage[margin=1in]{geometry}
\usepackage{xcolor}
\usepackage{tikz}
\usepackage{tikz-cd}

\usetikzlibrary{shapes.geometric, arrows.meta, positioning}

\tikzstyle{startstop} = [ellipse, draw, align=center, minimum height=3em, minimum width=6em]
\tikzstyle{process} = [rectangle, draw, align=center, minimum height=3em, minimum width=6em]
\tikzstyle{decision} = [diamond, draw, align=center, aspect=2, minimum height=2em, minimum width=4em]
\tikzstyle{inout} = [draw,
    shape=rectangle,
    minimum height=1.5em,
    minimum width=3em,
    shape border rotate=0,
    xslant=0.3, 
    align=center
]

\tikzstyle{labelbox} = [font=\small, align=center]

\usetikzlibrary{decorations.pathreplacing}
\usetikzlibrary{patterns,external}
\usetikzlibrary{calc}
\usetikzlibrary{shapes.multipart}
\usetikzlibrary{fit,positioning}

\usepackage[numbers,sort]{natbib}

\usepackage[linesnumbered,commentsnumbered,ruled,vlined]{algorithm2e}

\definecolor{fondo}{rgb}{0.898,0.996,0.898}

\pgfkeys{/tikz/.cd,
  K/.store in=\K,
  K=1   
   }

\title{How to recover a permutation group amidst errors}
\author{Taylor Brysiewicz and Juhee Kim}

\newcommand{\I}{{\mathbb{I}}}

\newcommand{\mydef}[1]{{\color{blue}#1}}
\newcommand{\mycomment}[1]{}

\newcommand{\iid}{i.i.d.~}

\usepackage[all,cmtip]{xy}

\renewcommand{\tilde}[1]{\widetilde{#1}}
\renewcommand{\bar}[1]{\overline{#1}}

\theoremstyle{definition}
\newtheorem{theorem}{Theorem}[section]

\newtheorem{lemma}[theorem]{Lemma} 
\newtheorem{corollary}[theorem]{Corollary} 
 
\newtheorem{example}[theorem]{Example}
\newtheorem{remark}[theorem]{Remark}
\newtheorem{proposition}[theorem]{Proposition}

\renewcommand{\Pr}{\textrm{Pr}}

\keywords{\hspace{-8pt} permutation groups; probabilistic group theory; randomized algorithms; monodromy groups}

\begin{document}
\begin{abstract}
We consider the problem of recovering a permutation group $G \leq S_n$ from an error-prone sampling process $X$.  We model $X$ as an $S_n$-valued random variable, defined as a mixture of the uniform distributions on $G$ and $S_n$. Our suite of tools recovers properties of $G$ from $X$ and bolsters our main method for recovering $G$ itself. Our algorithms are motivated by the numerical computation of monodromy groups, a setting where such error-prone sampling procedures occur organically. 
\end{abstract}

\maketitle

\vspace{-16pt}

\section{Introduction}
We consider the problem of recovering a permutation group $G \leq S_n$ from a sampling process which returns a uniform random element of $G$ with probability $1-p$ and otherwise returns a uniform element of $S_n$.  Outcomes of this process are those of an $S_n$-valued random variable $X=X(G,p)$:
\begin{align*}
\Pr(X = \sigma) &= \begin{cases} 
      (1-p)\frac{1}{|G|}+p\frac{1}{n!} &  \sigma  \in G \\
      p\frac{1}{n!} & \sigma \not\in G 
   \end{cases} .
\end{align*}

One approach for recovering $G$ from $X$ is to take $k$ independent observations $X_1,\ldots,X_k$ of $X$ and return the group they generate.  We call this the \emph{naive group recovery algorithm} (\autoref{alg:naivealgorithm}). Used in practice  to heuristically compute monodromy groups using numerical algebraic geometry \cite{Asante2025,Asante, DuffMinimal} it requires a choice of $k$ and is extremely sensitive to errors: if any observation $X_i$ is not in $G$ then the answer is incorrect. Worse, Luczak and Pyber's result (\autoref{prop:luczakpyber}) implies that wrong answers are often  \textit{very} wrong, in that the returned group is likely a   \textit{giant}:  $A_n$ or $S_n$.

The naive algorithm is  randomized, with some probability $\gamma(G,p,k)$ of success. Thus, it may be \emph{amplified} by repeating it $N$ times and returning the mode of the output (\autoref{alg:amplifiednaivegrouprecovery}). This is the \textbf{N}a\textbf{i}ve \textbf{A}mplified \textbf{G}roup \textbf{R}ecovery \textbf{A}lgorithm (\textbf{NiAGRA}). When $\gamma(G,p,k)>\frac 1 2$, amplification recovers $G$ with arbitrarily high confidence as $N \to \infty$. Thus, we give two error detection techniques which improve $\gamma(G,p,k)$ beyond $\frac 1 2$ if $G$ has some property $\mathcal Q$ or all $\sigma \in G$ have some property $\mathcal P$:
\begin{enumerate}
\item[] \hspace{-30pt} \textit{$\mathcal P$ Sample Error Detection:} Observe $X$ until the observed permutation has property~$\mathcal P$.  
\item[] \hspace{-26pt} \textit{$\mathcal Q$ Group Error Detection:} Run naive group recovery until the output has property $\mathcal Q$.
\end{enumerate}
If no properties are known, \textit{a priori}, our hypothesis tests can find one, using only observations of~$X$:
\[\begin{array}{lll}
\textbf{Giant Test} & \autoref{alg:gianttest} & \text{Determines if } G \text{ is } A_n \text{ or } S_n \text{ (i.e., a \textit{giant})} \\
\textbf{Subgroup Test} & \autoref{alg:subgrouptest} & \text{Determines if } G\leq H \text{ for given }H\\
\quad \text{       Alternating Test} & \,\,\text{ with }H=A_n & \text{Determines if } G \text{ is contained in } A_n \\
\quad \text{       Block Test} &\,\, \text{ with }H=\textrm{Wr}(\mathcal B)& \text{Determines if } \mathcal B \text{ is a block structure of } G \\
\quad \text{       Orbit Refining Test} &\,\, \text{ with }H=S_{\Delta} & \text{Determines if the orbits of } G \text{ refine } \Delta \\
\textbf{$k$-Transitivity Test} & \autoref{alg:transitivitytest} & \text{Determines if } G \text{ is $k$-transitive} \\
\textbf{Orbit Agreement} &  \autoref{alg:OrbitAgreement} & \text{Determines if } G.i=G.j \\
\quad \text{ Single Orbit Recovery} & \autoref{alg:SingleOrbitRecovery} & \text{Determines the orbit } G.i \text{ of } G \text{ containing } i \\
\quad \text{ Orbit Recovery} & \autoref{alg:OrbitRecovery} & \text{Determines the orbits of } G \\
\quad \text{ Orbit Confirmation} & \autoref{alg:OrbitConfirmation} & \text{Determines if } \Omega \subseteq [n] \text{ is an orbit of } G \\
\quad \text{ Heuristic Orbit Recovery} & \autoref{alg:HeuristicOrbitRecovery} & \text{Determines orbits of } G \text{ heuristically}\\
\quad \text{ Block Recovery} & \autoref{alg:BlockRecovery} & \text{Determines all minimal block structures of } G \\
\quad \text{ Primitivity Test} & \autoref{alg:PrimitivityTest} & \text{Determines if } G \text{ is primitive} \\
\end{array}\]

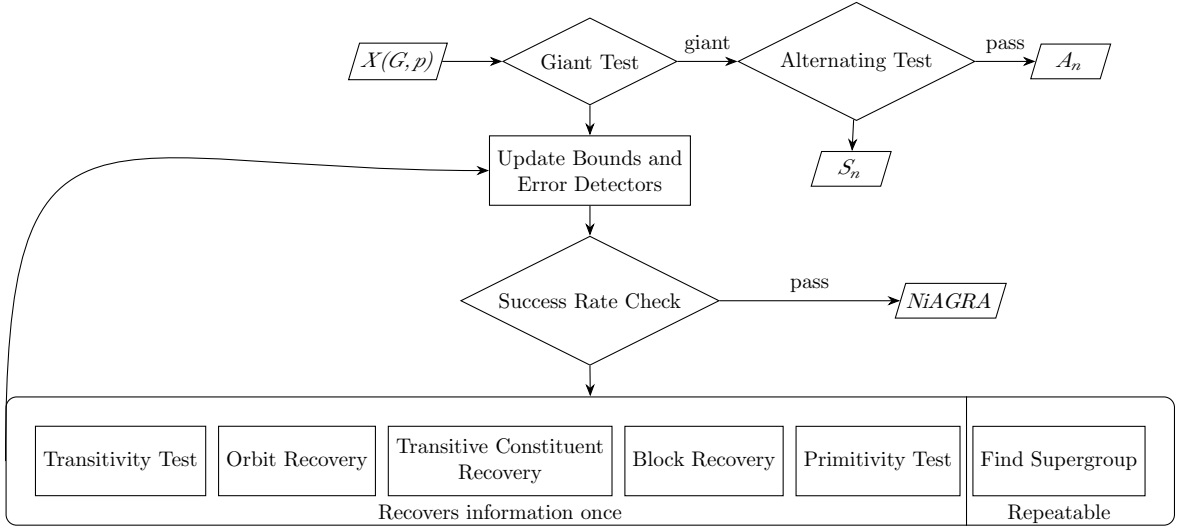
\begin{figure}[!htpb]
\begin{tikzpicture}[scale=0.8, transform shape, node distance=0.5cm and 1cm,  every node/.style={font=\small}, >=Stealth]

\node (input) [inout] {$X(G,p)$};
\node (giant) [decision, right=of input] {Giant Test};

\node (alt) [decision, right= of giant] {Alternating Test};

\node (An) [inout, right= of alt] {$A_n$};
\node (Sn) [inout, below= of alt] {$S_n$};

\draw[->] (input) -- (giant);

\draw[->] (giant) -- node[above]{giant} (alt);
\draw[->] (alt) -- node[above]{pass} (An);
\draw[->] (alt) -- (Sn);

\node (update) [process, below=of giant] {Update Bounds and\\ Error Detectors };
\node (success) [decision, below=of update] {Success Rate Check};

\draw[->] (giant) -- (update);
\draw[->] (update) -- (success);

\node (niagra) [inout, right=3cm of success] {NiAGRA};
\draw[->] (success) -- node[above]{pass} (niagra);

\node (t1) [process, below=1cm of success, xshift=-7.8cm] {Transitivity Test};
\node (t2) [process, right=0.2cm of t1] {Orbit Recovery};
\node (t3) [process, right=0.2cm of t2, label=below:{Recovers information once}] {Transitive Constituent\\ Recovery};
\node (t4) [process, right=0.2cm of t3] {Block Recovery};
\node (t5) [process, right=0.2cm of t4] {Primitivity Test};
\node (t6) [process, right=0.2cm of t5, label=below:{Repeatable}] {Find Supergroup};

\node (converge) [below=of t3] {};

\node (tests) [draw, rounded corners, fit=(t1)(t2)(t3)(t4)(t5)(t6),
               inner sep=1em] {};
               
               \draw
  ($(t6.west |- tests.north)  + (-1mm,0)$) --
  ($(t6.west |- tests.south) + (-1mm,0)$);

\draw[->] (success.south) to[] (tests.north);

\draw[->] (tests.west) to[out=90, in=180, looseness=1.8] (update.west);

\end{tikzpicture}
\caption{A flowchart of our main algorithm, \autoref{alg:main_algorithm}.}
\label{fig:flowchart}
\end{figure}

Each hypothesis test discovers a property about $G$ with some confidence level, which may then be converted into a sample error detector $\mathcal P$ or group error detector $\mathcal Q$. Each involves approximating the sample mean of a random variable $\mathcal T$ via sampling. This sample statistic  differentiates whether $G$ belongs to some ``Family $A$'' or another ``Family $B$'' of permutation groups. Our four main tests are summarized in \autoref{tab:main_tests}. The last row indicates the limiting feature or computation in the test for adversarial or large examples. For example, the transitivity test requires a very accurate approximation of an expected value, and thus a large sample size, to differentiate the group families.

\begin{table}[!htpb]
\begin{tabular}{|r||l|l|l|l|} \hline 
Name of Test &  GiantTest & SubgroupTest$_H$ & TransitivityTest$_k$ & OrbitAgreement$_{i,j}$ \\ \hline 
Test Var $\mathcal T$ & $\langle X_1, X_2 \rangle $ giant & $X \in H$ & $\textrm{Fix}_k(X)$ & $X(i)=j$\\
Family A  & $G \in \{A_n,S_n\}$ & $G \leq H$ & $G$ is $k$-transitive & $G.i = G.j$\\
Family B  & $G \not\in \{A_n, S_n\}$ & $G \not\leq H$ & $G$ is not $k$-transitive& $G.i \neq G.j$ \\
$\mathbb{E}(\mathcal T|A)$ & $> \mathcal L(n,p)$& $=1-p+\frac{p|H|}{n!}$ & $=1$ & $=\frac{1}{|G.i|} \geq \frac{1}{n}$ \\
$\mathbb{E}(\mathcal T|B)$  & $<\mathcal U(n,p)$& $= \frac{1-p}{2}+\frac{p|H|}{n!} \leq \frac{1}{2}$ & $>2-p$ & $=\frac{p}{n}\leq \frac{1}{n}$ \\
Theory & Dixon & LaGrange & Burnside &  \autoref{cor:movement_separation} \\
See Section & \autoref{secsec:giant_test} & \autoref{secsec:subgroup_test} & \autoref{secsec:transitivity_test} & \autoref{secsec:orbit_recovery} \\
Challenge & Requires $p<b_n$ & $H$ membership & Small $\mathbb{E}$ & Small $\mathbb{E}$ \\  \hline 
\end{tabular}
\caption{Summary of our four main tests for obtaining properties of $G$.}
\label{tab:main_tests}
\end{table}

\vspace{-3pt}

Beyond hypothesis tests, we use the following algorithms:
\[\begin{array}{lll}
\textbf{Find Supergroup} & \autoref{alg:findsupergroup} & \text{Iteratively builds a supergroup of }G \\
\textbf{Success Rate Check} & \autoref{alg:successratecheck} & \text{Bounds the success rate of }\autoref{alg:naivealgorithm} \\
\textbf{Transitive Constituent Recovery} & \autoref{alg:TransitiveConstituentRecovery} & \text{Recovers the action of } G \text{ on an orbit}
\end{array}\]
\textit{Find Supergroup} builds-up a supergroup $H$ of $G$ one generator at a time, only stopping when the subgroup test passes for $G \leq H$. The group property $\mathcal Q \colon G \leq H$ or permutation property $\mathcal P\colon \sigma \in H$ may then be used as an error detector. \textit{Success Rate Check} bounds the success rate of the algorithm which is amplified in \textbf{NiAGRA} using known bounds on $p$, $|G|$, and $|\{\sigma \in S_n \mid \mathcal P(\sigma)\}|$. When bounded above $\frac 1 2$ amplification is guaranteed to succeed with arbitrary confidence.  \textit{Transitive constituent recovery}   calls our main algorithm, on sampling procedures derived from $X$,  to compute the \textit{transitive constituents} $G^{\Delta_1},\ldots,G^{\Delta_m}$ of $G$, that is,  the actions of $G$ on its orbits $\Delta=(\Delta_1,\ldots,\Delta_m)$.

These tests and routines are the components of our main algorithm \autoref{alg:main_algorithm}, summarized in \autoref{fig:flowchart}, which culminates in \textbf{NiAGRA}. 
The validity of each subroutine rests upon results from probabilistic group theory, as outlined in \autoref{sec:background}. 
In \autoref{sec:permutation_sampling} we discuss our model $X$ for error-prone permutation sampling as well as sample error detection. In \autoref{sec:naive_group_recovery}, we analyze the naive group recovery algorithm and explain how to combine it with  sample error detection, group error detection, and amplification. 
\autoref{sec:property_recovery} contains our property recovery suite.  We obtain rigorous bounds on the confidence levels of these algorithms,  derived from worst-case-scenario groups and Hoeffding's inequality. 
In \autoref{sec:group_recovery} we give our main algorithm for recovering $G$. It is a combination of the tests developed in \autoref{sec:property_recovery} and the randomized framework of \autoref{sec:naive_group_recovery}.

 \autoref{sec:experiments} contains our experiments. We give strong evidence that repeating \textit{Find Supergroup}  with group error detector $\mathcal Q \colon ``G \text{ is a transitive non-giant}"$ suffices to determine $G$, when applicable. Additionally, we run an experiment which compares three of our main algorithms (the \textit{giant test}, the \textit{transitivity test}, and the non-adaptive version of the \textit{heuristic orbit recovery} algorithm)  when applied to \textit{fixed} sets of $\textrm{i.i.d.}$ observations of $X$. We perform this experiment for all subgroups of $S_{10}$. Our results suggest that, even for high error rates,  error-detected success rates exceed $\frac 1 2$ and so these algorithms may be amplified to arbitrary confidence. 

\noindent \textbf{Motivation:} The motivation for this work is the numerical computation of monodromy groups. In \cite{Sottile}, the authors describe how to compute monodromy groups of branched covers using numerical homotopy continuation, the core method in numerical algebraic geometry \cite{WhatIsNAG}. The algorithm works by first restricting to a general line in the base space, computing the branch locus, and then lifting generators for the fundamental group of the base without the branch locus to compute generators for the monodromy group.  In practice, computing branch loci can be difficult for large problems, and certifying that the numerically computed lifts are correct can be expensive (e.g. see  \cite{Telen}). 

Naive group recovery is a tempting alternative to this approach  due to its computational feasibility. Taking random loops in the base space of the branched cover induces a mysterious distribution on the fundamental group of the base, pushed-forward to a more mysterious one on the monodromy group $G$. Like $X(G,p)$, this  distribution  is subject to error with some rate $p$ (see \cite{HauensteinProbability} for one model for $p$) but unlike $X(G,p)$, it is \textit{not} a mixture of uniform distributions \cite{BCD}. The propensity for errors indicates group-theoretic fragility in this heuristic algorithm, which our work addresses.   In the language of the \textit{monodromy assumptions} (see \cite{MonodromyCoordinates}), our work shows that the third monodromy assumption of perfect path-lifting is not required to reliably recover monodromy groups numerically. 

\noindent \textbf{Code:} Implementations of the algorithms in this manuscript may be found in \href{https://github.com/kate0618kim/GroupRecovery.jl}{this github repository}.

\section*{Acknowledgements}
Both authors are supported by an NSERC Discovery Grant (RGPIN-2023-03551). The authors thank Colva Roney-Dougal for generous and insightful discussions about probabilistic group theory.

\section{Preliminaries on permutation groups and probability}
\label{sec:background}
This section provides the necessary background on permutation groups and probabilistic group theory. The former topic is that of textbooks (e.g. \cite{DixonPermutationGroups}) whereas we pull results on probabilistic group theory from several papers.  Throughout, we let $\mydef{S_n}$ denote the \mydef{symmetric group} on $\mydef{[n]} = \{1,2,\ldots,n\}$ and we call any subgroup  of $S_n$  a \mydef{permutation group} of \mydef{degree} $n$.

\subsection{Orbits and Transitivity}
 A permutation group $G \leq S_n$  naturally acts on $[n]$ with orbits $\mydef{\Delta}=(\mydef{\Delta_1}, \cdots, \mydef{\Delta_m})$ so that $\Delta_1 \sqcup \cdots \sqcup \Delta_m = [n]$. Given a partition $\mydef{\lambda}=(\lambda_1,\ldots,\lambda_m)$ of $[n]$ into $m$ parts of sizes $n_1,\ldots,n_m$, we write
$\mydef{S_{\lambda}} =S_{\lambda_1} \times S_{\lambda_2} \times \cdots \times S_{\lambda_m}\cong S_{n_1} \times S_{n_2} \times \cdots \times S_{n_m}$
for the subgroup of permutations respecting $\lambda$. This group has order  $n_1!n_2! \cdots  n_m!$ and is called a  \mydef{Young subgroup} of $S_n$. Notably, $S_{\Delta}$ is the smallest Young subgroup containing $G$.

A permutation group $G$ is \mydef{transitive} if it has a unique orbit. Trivially, any permutation group acts transitively on each of its orbits. We write $\mydef{G^{\Delta_i}}$ for the permutation group in $S_{\Delta_i} \cong S_{n_i}$ obtained by restricting the action of $G$ to the orbit $\Delta_i$. The groups $\{G^{\Delta_i}\}_{i=1}^m$ are the \mydef{transitive constituents} of $G$. As the name suggests, each transitive constituent is a transitive group.   We remark, however, that the transitive constituents of an intransitive group do not necessarily determine it. 

Given  $\sigma \in S_n$, write $\mydef{\textrm{Fix}(\sigma)}$ for the number of fixed points of $\sigma$. Burnside's lemma relates the number of fixed points of permutations in $G$ with the number of orbits of $G$.

\begin{proposition}[Burnside's Lemma]
\label{lem:Burnside}
Fix $G \leq S_n$ with $m$ orbits, then
\[
\frac{1}{|G|}\sum_{g \in G} \textrm{Fix}(g) = m.
\]
\end{proposition}

The action of any group $G\leq S_n$ on $[n]$ naturally extends to an action on distinct ordered $k$-tuples $\{(i_1,\ldots,i_k) \in [n]^k \mid i_1,\ldots,i_k \text{ are distinct}\}$, represented by the permutation groups $\mydef{G^{(k)}} \leq S_{k!{{n}\choose{k}}}$. Note that $G \cong G^{(k)}$ are isomorphic as abstract groups for all $k$, but not as permutation groups. The group $G$ is called \mydef{$k$-transitive} if $G^{(k)}$ is transitive. Since $G^{(k)}$ is a permutation group, one may apply Burnside's lemma to $G^{(k)}$ to characterize $k$-transitivity of $G$.
\begin{corollary}
Fix $G \leq S_n$. Then $G$ is $k$-transitive if and only if
\[
\frac{1}{|G^{(k)}|}\sum_{g \in G^{(k)}}\textrm{Fix}(g) = 1.
\]
\end{corollary}

\subsection{Blocks and Primitivity}
A transitive group is called \mydef{imprimitive} if it preserves a  partition $\mydef{\mathcal B} = \{\mydef{B_1}, \cdots,  \mydef{B_k}\}$ of $[n]$ which is non-trivial ($k$ is neither $1$ nor $n$). That is, for every $\sigma$ in the group, $\sigma(B_i)=B_j$ for some $j$. Non-imprimitive transitive groups are called \mydef{primitive}. Any non-trivial partition $\mathcal B$ preserved by $G$ is called a \mydef{block structure} of $G$. The parts of $\mathcal B$ are called \mydef{blocks}. Block structures of $G$ which are minimal with respect to refinement are \mydef{minimal block structures} of $G$. 

A transitive group acts transitively on the set of blocks comprising any of its block systems. In particular, all blocks of a particular block structure of a transitive group have the same size.  If $\mathcal B=\{B_1,\ldots,B_k\}$ is a block structure for a transitive subgroup of $S_n$, we write  $\mydef{\textrm{Wr}(\mathcal B)} \cong S_{\frac{n}{k}} \wr S_{k}$ for the \mydef{wreath product} of permutations which preserve the partition $\mathcal B$.  The order of $\textrm{Wr}(\mathcal B)$ is  $(\frac{n}{k}!)^{k}  k!$.  In analogy with Young subgroups being the subgroups of permutations respecting an orbit partition, wreath products are the subgroups of permutations respecting a block structure. Note that $2$-transitivity implies primitivity, but the converse is false (see 
\autoref{ex:primitive_not_two_transitive}).
The following proposition is an exercise (e.g. \cite[Exercise 5.2.8]{DixonPermutationGroups}).
\begin{proposition}
\label{prop:maximals}
Each maximal proper subgroup of $S_n$ other than $A_n$ is one of the following forms
\begin{itemize}
\item (Intransitive) A maximal  Young subgroup $S_{k} \times S_{n-k}$ for some $1 \leq k \leq n$. 
\item (Imprimitive) A maximal   wreath product $S_{a} \wr S_b$ with $a\cdot b = n$ and $a,b \neq 1$.
\item (Primitive) A proper maximal primitive group.
\end{itemize}
\end{proposition}
\noindent The O'nan-Scott theorem   gives a more detailed taxonomy of maximal primitive groups \cite[Ch. 4.8]{DixonPermutationGroups}.

\subsection{Jordan, giant, and primitive permutations}
The orders of maximal non-alternating subgroups of $S_n$ are significantly bounded  (see \autoref{lem:Saxl_and_friends}). This fact affords the groups $A_n$ and $S_n$ the title of \mydef{giant} permutation groups, terminology first used in   \cite{BabaiLuks}. Given  $\sigma \in S_n$, we say:
\begin{itemize}
\item $\sigma$ is \mydef{Jordan} if the only primitive groups containing it are giants.
\item  $\sigma $ is \mydef{giant} if the only transitive groups containing it are giants \\ \quad \quad \quad \quad {\color{white}{.......}}Also called \emph{strongly primitive} (e.g. in \cite{Arajuo}) in the literature.
\item $\sigma$ is \mydef{primitive} if the only transitive groups containing it are primitive. 
\end{itemize} 
Permutations which are not Jordan (resp. giant, primitive) are called \mydef{non-Jordan} (resp. \mydef{non-giant} \mydef{imprimitive}).
\begin{example}

There are $22$ cycle types in $S_8$. \autoref{fig:S8CycleTypes} shows the twelve which are Jordan, giant, or primitive, and  \autoref{tab:jordan_giant_primitive} gives the first few proportions of Jordan, giant, and primitive permutations. 
The proportion of elements which are Jordan is $1.0$ whenever there exist no primitive groups in $S_n$. We remark that this property is the rule, and not the exception: for most $n$, there are no primitive groups of order $n$ (see \cite{CameronNeumannTeague,AffPrimitive}). Similarly, the proportion of elements which are primitive is $1.0$ whenever there are no imprimitive groups of order $n$, which is the case exactly when $n$ is prime.

\begin{table}[!htpb]
$\begin{array}{l}
\begin{tabular}{|c|*{12}{p{2.2em}|}}
\hline
$n$ & 1 & 2 & 3 & 4 & 5 & 6 & 7 & 8 & 9 & 10 & 11 & 12 \\
\hline
\textrm{Jordan} & 1.00 & 1.00 & 1.00 & 1.00 & 0.417 & 0.368 & 0.468 & 0.426 & 0.501 & 0.660 & 0.690 & 0.602  \\
\hline
\textrm{Giant} & 1.00 & 1.00 & 1.00 & 0.333     & 0.417 & 0.000 & 0.468 & 0.200 & 0.315 & 0.254 & 0.690 & 0.168  \\
\hline
\textrm{Primitive} & 1.00 & 1.00 & 1.00 &0.333   & 1.00  & 0.200 & 1.00  & 0.343 & 0.543 & 0.316 & 1.00  & 0.259  \\
\hline
\end{tabular}  \\ \\ 
\begin{tabular}{|c|*{12}{p{2.2em}|}}
\hline
$n$ & 13 & 14 & 15 & 16 & 17 & 18 & 19 & 20 & 21 & 22 & 23 & 24 \\
\hline
\textrm{Jordan} & 0.773 & 0.791 & 0.910 & 0.853 & 0.810 & 0.843 & 0.885 & 0.861 & 0.920 & 0.977 & 0.896 & 0.858 \\
\hline
\textrm{Giant} & 0.773 & 0.313 & 0.467 & 0.325 & 0.810 & 0.247 & 0.885 & 0.318 & 0.567 & 0.467 & 0.896 & 0.276 \\
\hline
\textrm{Primitive} & 1.00  & 0.390 & 0.477 & 0.393 & 1.00  & 0.306 & 1.00  & 0.370 & 0.585 & 0.478 & 1.00  & 0.319 \\
\hline
\end{tabular} \\ \\
\begin{tabular}{|c|*{12}{p{2.2em}|}}
\hline
$n$ & 25 & 26 & 27 & 28 & 29 & 30 & 31 & 32 & 33 & 34\\
\hline
\textrm{Jordan} & 0.920 & 0.933 & 0.943 & 0.938 & 0.927  & 0.911 & 0.921 & 0.889 & 0.953 & 1.00 \\
\hline
\textrm{Giant} & 0.745 & 0.506 & 0.615 & 0.426 & 0.927  & 0.310 & 0.921 & 0.458 & 0.656 & 0.552\\
\hline
\textrm{Primitive} & 0.791 & 0.507 & 0.657 & 0.426 & 1.00  &  0.345 & 1.00 & 0.490 &0.672 & 0.552 \\
\hline
\end{tabular}
\end{array}$
\caption{The proportion of permutations which are Jordan, giant, and primitive respectively. As $n \to \infty$ each proportion approaches $1$.}
\label{tab:jordan_giant_primitive}
\end{table}

\begin{remark}[Computing Jordan, Giant, and Primitive proportions]
The observation of \autoref{prop:maximals} along with the database of primitive groups in GAP \cite{GAP} allows us to compute the numbers of \autoref{tab:jordan_giant_primitive} exactly. To compute the primitive proportion, we count the proportion of permutations which are imprimitive by collecting all cycle types in all maximal wreath products and summing the sizes of the corresponding conjugacy classes. To calculate the Jordan proportion, we collect the cycle types in primitive groups and sum their conjugacy class sizes to obtain the proportion of non-Jordan permutations. The non-giants are those which are not Jordan or are  imprimitive, and so the giant proportion is computed as a byproduct of the earlier computations.
\end{remark}

\begin{figure}[!htpb]
\includegraphics[scale=0.34]{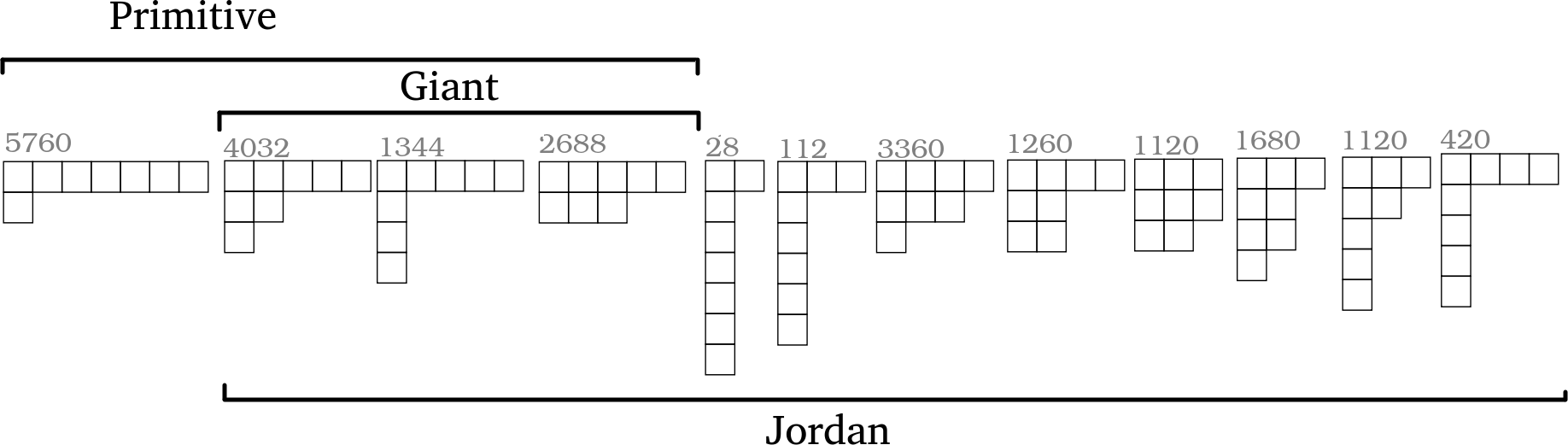}
\caption{Young diagrams of twelve of the $22$ cycle types of elements of $S_8$ categorized by being giant, Jordan, or primitive, and decorated with the number of permutations of that cycle type. Other cycle types are imprimitive and non-Jordan.}
\label{fig:S8CycleTypes}
\end{figure}
\end{example}

Further justifying the use of the terminology \textit{giant}, Luczak and Pyber \cite{LuczakPyber} proved that the proportion of $S_n$ comprised of transitive permutation groups (i.e. non-giant permutations) is vanishingly small as $n \to \infty$ (see \autoref{prop:luczakpyber}). 

\subsection{Random permutations}
In probabilistic group theory, one studies questions about random permutations in a permutation group, often with respect to the uniform distribution on $S_n$. Fundamentally combinatorial, these questions often rely on enumerative results. Conversely, enumerative results are easily translated into statements about probabilities.
\begin{lemma}
\label{lem:Saxl_and_friends}
If $G$ is a non-giant, then its order is bounded as follows.
\begin{itemize}
\item (Intransitive) The largest proper Young subgroup in $S_n$ is $S_1 \times S_{n-1}$ which has order $(n-1)!$ and index $n$. This is the smallest index of a non-giant in $S_n$

\item (Imprimitive) The largest proper wreath product in $S_{2d}$ is $S_d \wr S_2$ which has order $2d!^2$ and index $2{{2d}\choose{d}}^{-1}$ attaining  the smallest index among transitive non-giants in $S_{2d}$. This bound on the largest proper wreath product holds for $n=2d+1$ and $d>4$ as well.

\item (Primitive) Any primitive non-giant in $S_n$ has order at most $4^n$ \cite{Praeger}.
\end{itemize}

\end{lemma}
\begin{proof}
This result follows from the trichotomy of \autoref{prop:maximals}, the orders of Young subgroups and natural wreath products, and a non-trivial result by Saxl and Praeger \cite{Praeger}.
\end{proof}
\begin{lemma}
\label{lem:prop_orbit_wreath_respect}
The probability that $\sigma \in S_n$ is in the Young subgroup $S_{\Delta} \cong S_{n_1} \times \cdots \times S_{n_m}$ is 
\[
\Pr(\sigma \in S_\Delta) = [S_n:S_{\Delta}]^{-1}  = \frac{|S_\Delta|}{|S_n|} = \frac{n_1! \cdot n_2! \cdots n_m!}{n!} = {{n}\choose{n_1,n_2,\ldots,n_m}}^{-1} \leq \frac{1}{n}
\]
The probability that  $\sigma$ belongs to the wreath product $\textrm{Wr}(\mathcal B)$ for $\mathcal B = \{B_1,\ldots,B_{k}\}$ is 
\[
\Pr(\sigma \in \textrm{Wr}(\mathcal B)) = [S_n:\textrm{Wr}(\mathcal B)]^{-1} =  \frac{|\textrm{Wr}(\mathcal B)|}{|S_n|} = \frac{\left(\frac{n}{k}!\right)^{k}k!}{n!} \leq 2\left(\frac{1}{\sqrt{2}}\right)^n.
\]
The maxima are achieved by the groups $S_1 \times S_{n-1}$ and $S_{n/2} \wr S_2$, when $n$ is even, respectively.
\end{lemma}
\begin{proof}
In each result, the equalities are obtained through basic counting. The first inequality holds since $S_1 \times S_{n-1}$ is smallest index Young subgroup of $S_n$. The second inequality holds since $S_{\frac n 2} \wr S_2$ is the smallest index wreath product in $S_{n}$, and so we bound
\[
\Pr(\sigma \in \textrm{Wr}(\mathcal B)) = [S_{n}:S_{n/2} \wr S_2]^{-1} = \frac{(n/2)!^2 \cdot 2}{n!} = 2 {{n}\choose{n/2}}^{-1}
\] which is bounded above by $2\sqrt{2}^{-n}$ since  ${{n}\choose{n/2}}$ is bounded below by $(\sqrt{2})^n$.
\end{proof}

The celebrated result of Luczak and Pyber states that the proportion of giant permutations approaches $1$ as $n \to \infty$, contrary to what one may expect from the first few values of the giant proportions in \autoref{tab:jordan_giant_primitive}. 
In other words, for large $n$, the union of all transitive non-giant permutation groups comprises a vanishingly small proportion of $S_n$. Explicit bounds are obtained in \cite{Eberhard}.
\begin{proposition}[Luczak Pyber \cite{LuczakPyber}]
\label{prop:luczakpyber}If $\sigma$ is a random permutation in $S_n$, then
\[
\lim_{n \to \infty} \Pr(\sigma \text{ is a giant}) = 1.
\]
\end{proposition}
\noindent Dixon's conjecture, proven by Babai \cite{Babai}, says that two random permutations likely generate~$S_n$. 
\begin{proposition}[Dixon's conjecture \cite{Babai}] \label{prop:dixon}
If $\sigma_1,\sigma_2$ are random permutations in $S_n$, then
\[
\Pr(\langle \sigma_1,\sigma_2 \rangle \text{ is a giant}) = 1 - \frac{1}{n} + O(n^{-2}).
\]
\end{proposition}
Effective bounds on the probability of \autoref{prop:dixon} were obtained in \cite{Maroti} and improved upon by Morgan and Roney-Dougal in \cite{Morgan}, which also covers the alternating case. 
\begin{proposition}[Morgan and Roney-Dougal \cite{Morgan}]
\label{prop:morganroneydougal}
Let $G$ be $A_n$ or $S_n$ for $n \geq 5$ Then two random permutations $\sigma_1,\sigma_2$ in $G$ generate a giant with probability
\[
\mydef{\ell(n)} = 1-\frac{1}{n}-\frac{8.8}{n^2} <
\Pr(\langle \sigma_1,\sigma_2 \rangle \text{ is a giant}) < 1-\frac{1}{n}-\frac{0.93}{n^2} = \mydef{u(n)}.
\]
\end{proposition}
An important probability associated to a group $G$, and $k \in \mathbb{N}$, is the probability $\mydef{\varphi_k(G)}$ that $k$ independent random elements of $G$ generate it. Pak \cite{Pak} provides lower bounds on these values, obtained in the worst-case scenario by the regular group $\mathbb{Z}_2^M$, realized as a permutation group in $S_{2^M}$ generated by the coordinate reflections of the vertices of the cube~$[-1,1]^M$. 
\begin{lemma}[Pak]{\cite[Theorem 1.1]{Pak}}
\label{lem:Pak}
For any finite group $G$ with $M = \lceil \log_2(|G|) \rceil$ 
\[
\varphi_k(G) \geq \varphi_k(\mathbb{Z}_2^M) = \prod_{i=k-M+1}^{k} \left(1-\frac{1}{2^i}\right) > 1- \frac{8}{2^{k-M}}.
\]
Notably, $\varphi_M(G) > 1/4$ and $\varphi_{M+1}(G) > 1/2$.
\end{lemma}

\section{Error-prone permutation sampling and sample error detection}
\label{sec:permutation_sampling}
\subsection{Error-prone permutation sampling}
For any subset $S \subset S_n$, write $\mydef{U_S}$ for the $S_n$-valued random variable uniformly distributed on $S$: 
\[\Pr(U_S = \sigma) = \frac{1}{|S|}\I_{\sigma \in S} = \begin{cases} 
      \frac{1}{|S|} & \sigma \in S  \\
      0 & \text{ otherwise} 
   \end{cases} 
\] where $\mydef{\mathbb{I}_{*}}$ is the indicator function on the predicate $*$. 
Unlike the classical setting, we are interested in distributions on permutation groups which are not uniform.

Our main object of interest is the random variable $\mydef{X}=\mydef{X(G,p)}$ defined as a  mixture of $\mydef{Y}=U_G$ and $\mydef{Z}=U_{S_n}$ and depending on a fixed \mydef{error probability} we denote by $\mydef{p} \in [0,1]$.
Observations of  $X$ represent attempts to sample from $G$ via the procedure: 
\begin{center} \fbox{ \parbox{0.95\textwidth}{\textbf{Error-prone sampling procedure}\textit{ With probability $1-p$ return a uniform random   element of $G$ and otherwise return an \mydef{error} permutation uniformly sampled from $S_n$}}}
\end{center} 
 The probability mass function of $X$ is a convex combination of those of $Y$ and $Z$:
\begin{align}
\label{eq:pmfX}
\nonumber \Pr(X=\sigma) &= (1-p)\cdot \Pr(Y = \sigma)+p\cdot\Pr(Z=\sigma)\\ &=(1-p)\cdot \frac{1}{|G|} \I_{\sigma \in G} + p\cdot \frac{1}{n!}\\ \nonumber &= \begin{cases} 
      (1-p)\frac{1}{|G|}+p\frac{1}{n!} &  \sigma  \in G \\
      p\frac{1}{n!} & \sigma \not\in G 
   \end{cases} 
\end{align}
This random variable is a hidden mixture model with respect to a latent Bernoulli random variable \textrm{Ber}$(p)$ whose value determines whether to sample uniformly from $S_n$ or $G$. The value of \textrm{Ber}$(p)$ cannot be observed and so one cannot determine if an observation of $X$  in $G$ came from $Y$ or $Z$. Hence, we refer to observations of $X$ which do not belong to $G$ as \mydef{recognizable errors}. Since elements of $G$ have equal mass under the distribution of $Z$, we can rewrite $X$ as another mixture model: 
\begin{align}
\label{eq:pmfX_q}
 \nonumber \Pr(X=\sigma) = \left(1-p+\frac{p|G|}{n!}\right)\frac{1}{|G|} \I_{\sigma \in G} + \frac{p}{n!} \I_{\sigma \not\in G} 
&= (1-q)\frac{1}{|G|}\I_{\sigma \in G} + q \cdot \frac{1}{n!-|G|}\I_{\sigma \not\in G}.\\ \nonumber &= \begin{cases} 
      \frac{1-q}{|G|} &  \sigma  \in G \\
      \frac{q}{n!-|G|} & \sigma \not\in G 
   \end{cases} .
\end{align}
We refer to the value of $\mydef{q} = p(1-|G|/n!)$  as the \mydef{recognizable error probability} of $X$.

\subsection{Sample Error Detection}
\label{secsec:samplerrordetection} It is often the case that one has \emph{a priori} knowledge about the permutation group $G \leq S_n$ which can be used to reduce the error rate of $X$. Such is the case, for example, in the context of monodromy groups in algebraic geometry \cite{Sottile,Decomposable} one often knows that the group is transitive or primitive based on the geometry involved.

Let $\mydef{\mathcal P}$ be a property of permutations which holds for every element of $G$. Write $\mydef{S_{\mathcal P}}\subseteq S_n$ for the subset of permutations which have property $\mathcal P$. Note that $G \subseteq S_{\mathcal P}$.  We define $\mydef{X_{\mathcal P}}$ to be the random variable corresponding to the sampling procedure:
\begin{center}
 \fbox{ \parbox{0.9\textwidth}{\textbf{Error-prone sampling procedure with $\mathcal P$-error detection} \textit{repeatedly sample from $X$ until a permutation $\sigma$ is sampled which has property $\mathcal P$, then return $\sigma$}}}

\end{center} 
 Extending notation, write $\mydef{q_{\mathcal P}}$ for the probability that an observation of $X_{\mathcal P}$ does not belong to $G$.
\begin{lemma}
\label{lem:pmf_sample_ed}
The probability mass function of $X_{\mathcal P}$ is 
\begin{align*}\Pr(X_{\mathcal P} = \sigma) &= (1-q_{\mathcal P}) \cdot  \mathbb{I}_{\sigma \in G} + q_{\mathcal P}  \cdot  \mathbb{I}_{\sigma \in S_{\mathcal P}-G}
\end{align*}
where 
\[
q_{\mathcal P} = \frac{(|S_{\mathcal P}|-|G|)p}{n!-p(n!-|S_{\mathcal P}|)} = \frac{(B-A)p}{1-p(1-B)} \quad \quad A = |G|/n! \quad B = |S_{\mathcal P}|/n!
\]  The expected number of observations of $X$ required to obtain one of $X_{\mathcal P}$ is $\frac{n!}{n!-p(n!-|S_{\mathcal P}|)}=\frac{1}{1-p(1-B)}$. 
\end{lemma}
\begin{proof}
The process $X_{\mathcal P}$ can be thought of as a small Markov chain as illustrated in \autoref{fig:markov}. 
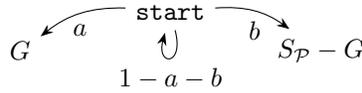
\begin{figure}[!htpb]
\begin{tikzpicture}[>=Stealth, node distance=0.5cm, every node/.style={font=\small}]
  \node[] (start) {\texttt{start}};
  \node[ below of=start, xshift=-2cm] (G) {$G$};
  \node[ below of=start, xshift=2cm] (SG) {$S_\mathcal{P} - G$};

  \draw[->] (start) to[bend right=20] node[below] {$a$} (G);
  \draw[->] (start) to[bend left=20] node[below] {$b$} (SG);
  \draw[->] (start) edge[loop below] node[below] {$1 - a - b$} (start);
\end{tikzpicture}
\caption{Markov chain for $X_{\mathcal P}$.}
\label{fig:markov}
\end{figure}

\noindent The probability $\mydef{a}$ of returning an element of $G$ on a single step of the Markov chain is
\[a= 1-q=1-p+\frac{|G|}{n!}p=1-p+Ap, \quad \text{ where } \mydef{A} = \frac{|G|}{n!}.\]
 The probability of returning an element of $S_{\mathcal P}-G$ is $\mydef{b}=(B-A)p$ where $\mydef{B} = \frac{|S_{\mathcal P}|}{n!}$. At each step the probability of returning something is $a+b$. Since the probability of returning an element of $G$ at a step is $a$, the probability of returning an element of $G$ at any time during the entire procedure is $\frac{a}{a+b}$. The expected number of repetitions necessary to return something is $\frac{1}{a+b}$. 
\end{proof}

\section{Naive group recovery, error detection, and amplification}
\label{sec:naive_group_recovery}
\subsection{Naive group recovery}
Our ultimate goal is to find $G$ using only observations of $X=X(G,p)$. The following algorithm,  used in practice (see \cite{Asante,Asante2025,DuffMinimal}), attempts to achieve this task by returning the permutation group generated by $k$ \iid samples $X_1,\ldots,X_k$ for some $k$. We refer to it as the  \mydef{naive group recovery algorithm} due to its simplicity and sensitivity to errors.

\begin{algorithm}[]
  \SetKwIF{If}{ElseIf}{Else}{if}{then}{elif}{else}{}
  \DontPrintSemicolon
  \SetKwProg{Naive Group Recovery}{Naive Group Recovery}{}{}
  \LinesNotNumbered
  \KwIn{$(X,k)$ \\ $\bullet$ The ability to sample from the distribution of $X$,\\  
  $\bullet$ A sample size $k$} 
  \KwOut{$G$} 
  \nl Sample $X_1,\ldots,X_k$ from the distribution of $X$ and \textbf{return} $\langle X_1,\ldots,X_k \rangle$\;
  \caption{Naive Group Recovery\label{alg:naivealgorithm}}
\end{algorithm}

\autoref{alg:naivealgorithm} is a randomized algorithm with some success rate  $\mydef{\gamma(G,p,k)}$.
\begin{lemma}
\label{lem:naiveprob}
The probability that \autoref{alg:naivealgorithm} returns $G$ is 
\[
\gamma(G,p,k) = (1-q)^k\varphi_k(G).
\]
If \autoref{alg:naivealgorithm} is applied to $X_{\mathcal P}$ then its success rate is $\mydef{\gamma_{\mathcal P}(G,p,k)}=(1-q_{\mathcal P})^k\varphi_k(G)$.
\end{lemma}
\begin{proof}
The returned group $\langle X_1,\ldots, X_k \rangle$ is $G$ only if each $X_1,\ldots,X_k$ belongs to $G$. Conditioning on each observation belonging to $G$, the probability that they generate $G$ is $\varphi_k(G)$ so
\[
    \Pr(\langle X_1,\ldots,X_k \rangle = G) = \varphi_k(G) \cdot \Pr(X_1,\ldots,X_k \in G)= \varphi_k(G) \cdot \Pr(X \in G)^k= \varphi_k(G) \cdot (1-q)^k.
   \] The analogous proof for $\gamma_{\mathcal P}(G,p,k)$ follows the same logic. 
\end{proof}
One drawback of \autoref{alg:naivealgorithm} is that the user must choose an appropriate $k$ for an unknown group $G$. There are bounds on $k$ which guarantee $\varphi_k(G)>1-\varepsilon$ and depend only on $n$ and $\varepsilon>0$ \cite[Corollary 1.4]{Colva}. Using $k$ larger than such a bound incurs the cost that $(1-q)^k$ becomes smaller. Conversely, as $k$ becomes small $\varphi_k(G)$ may decrease as well, possibly to zero. Pak's result \autoref{lem:Pak} gives a direct bound on $\varphi_k(G)$ via a worst-case scenario based on the order of $G$:
\[
\gamma(G,p,k) > (1-q)^k \varphi_k(\mathbb{Z}_2^M) >(1-p)^k\left(1-\frac{8}{2^{k-M}}\right) \quad \text{ where } M=\lceil \log_2(|G|)\rceil.
\] 

Combining \autoref{lem:Pak} with \autoref{lem:naiveprob} produces crucial threshold $\mydef{\omega_{\delta}(G,k)}$ on $q$ which guarantees that \autoref{alg:naivealgorithm} returns the correct answer with probability at least $\frac{1}{2}+\delta$. 
\begin{corollary}
\label{cor:prob_bounds_for_half_prob}
Fix $G \leq S_n$ and $\delta \in \left[0,\frac 1 2\right)$. The value $\gamma(G,p,k)$ is greater than $1/2+\delta$ if 
\begin{equation}
\label{eq:true_q_bound}
q <  1-\sqrt[k]{\frac{\frac{1}{2} + \delta}{\varphi_k(G)}}=:\mydef{\omega_{\delta}(G,k)}.
\end{equation}
A sufficient condition for this inequality to hold is 
\begin{equation}
\label{eq:pak_p_bound}
k>\log_2(|G|)+3\quad \text{ and }\quad p< 1-\sqrt[k]{\frac{\frac{1}{2} + \delta}{\left(1-\frac{8}{2^{k-\log_2(|G|)}}\right)}}.
\end{equation}
\end{corollary}

\begin{example}[The monodromy group of lines on a cubic surface] \label{ex:27lines}  Consider the group generated by 
\begin{align*}
\sigma_1 &= (1,10,13)(2,24,6)(3,17,11)(4,23,8)(5,26,25)(7,18,12)(9,20,16)(14,27,19)(15,21,22)\\
\sigma_2 &= (1,18,13,22,10,11)(2,4,21,27,9,15)(3,20,26,5,14,7)(6,25,23)(8,12,17,24,16,19).
\end{align*}
in $S_{27}$. 
This is the Weyl group $\mydef{W(E_6)}$ of $E_6$, occuring in the LMFDB database as Galois group  \href{https://beta.lmfdb.org/GaloisGroup/27T1161}{27T1161} \cite{lmfdb}. It is the \emph{monodromy group} of the classical \textit{problem of $27$ lines} (see \cite{CayleySalmon,JordanMonodromy}). 

\autoref{lem:naiveprob} states that the probability that \autoref{alg:naivealgorithm} succeeds on $W(E_6)$ is $\varphi_k(G)(1-q)^k$ which is  approximately $\varphi_k(G)(1-p)^k$ since $|W(E_6)|=51840 \ll 27!$. 
\autoref{lem:Pak} applied to $W(E_6)$ states that $M+1  = \lceil\log_2(51840) \rceil+1 =17$ uniform independent elements of $W(E_6)$  are more likely than not to generate it. In the error-prone setting, even with a small error probability like $p=0.01$, the confidence that $17$ elements sampled from $X$ belong to $W(E_6)$ is not very high: 
\[
\Pr(X_1,\ldots,X_{17} \in G)=(1-q)^{17} \approx (0.99)^{17} \approx 0.84.
\] 
In this instance, Pak's bound of $17$ is very conservative. \autoref{table:phiks} lists values of $\varphi_k(W(E_6))$ obtained via Monte Carlo approximations ($N=100,000$) in \texttt{Oscar.jl} \cite{OSCAR,OSCAR-book}. Note that $\varphi_2(W(E_6))>\frac 1 2$.

\begin{table}[!htpb]
{\hspace{-15pt}\footnotesize{
\begin{tabular}{|c|c|c|c|c|c|c|c|c|c|c|c|c|c|c|} \hline 
$k$ & $1$ & $2$ & $3$ & $4$ & $5$ & $6$ & $7$ & $8$ & $9$ & $10$ & $11$ & $ \cdots $ & $ 16$ & 17 \\ \hline \hline 
$\varphi_k(W(E_6))$ & $0.00$ & $0.667$ & $0.873$ & $0.935$ & $0.970$ & $0.988$ & $0.98$ & $0.993$ & $0.997$ & $0.999$ & $1.0$ & $\cdots$ & $1.0$  & $1.0$\\  \hline 
\end{tabular}
}}
\caption{Approximate values of $\varphi_k(W(E_6))$ for $k=1,\ldots,17$.}
\label{table:phiks}
\end{table}
\begin{figure}[!htpb]
\includegraphics[scale=0.2]{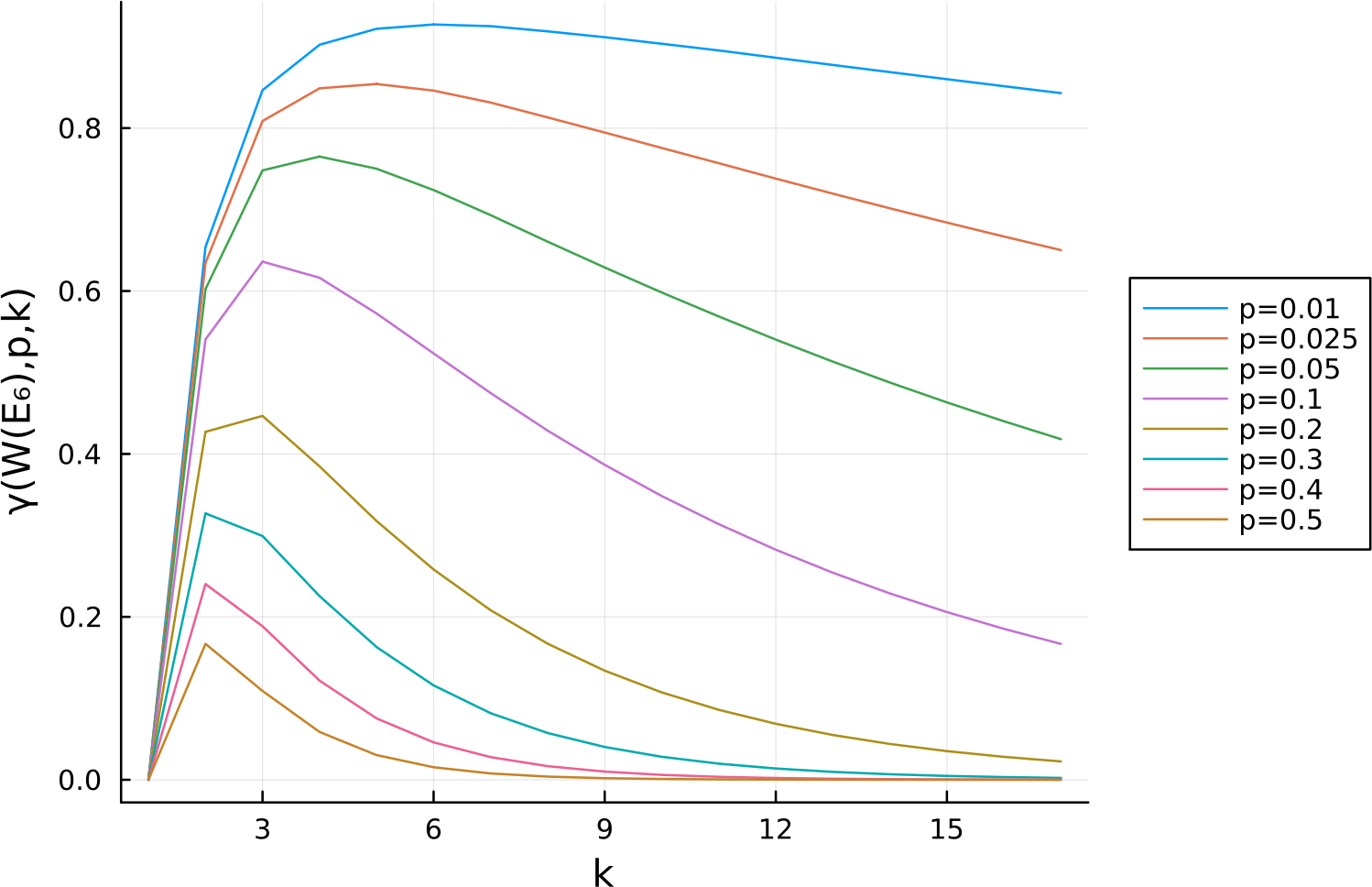}
\includegraphics[scale=0.2]{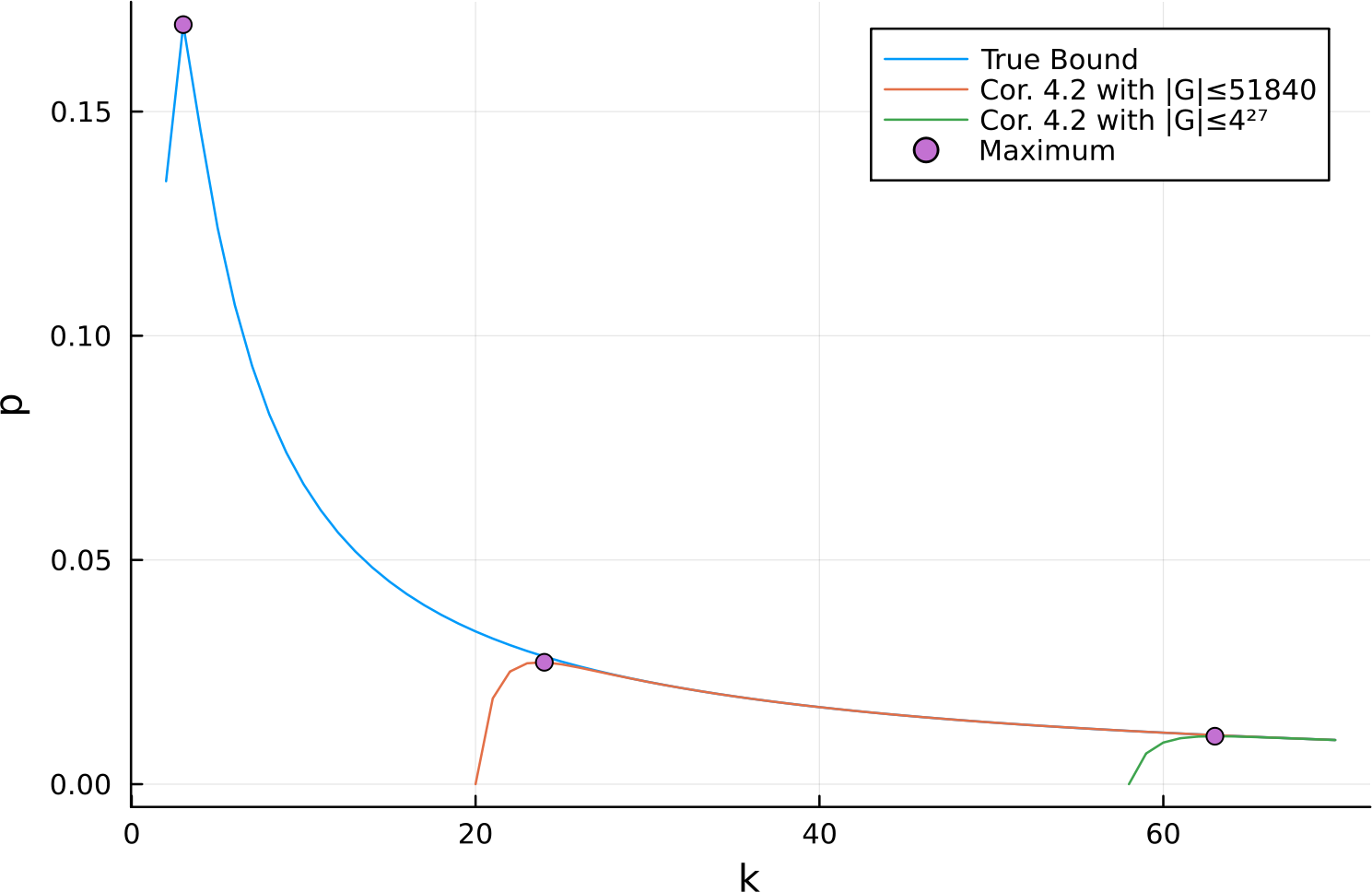}
\caption{(Left) The probabilities $\gamma(W(E_6),p,k)$ for various $p$ and $k$. (Right) Bounds on $p$ from \autoref{cor:prob_bounds_for_half_prob} for which $\gamma(W(E_6),p,k)>1/2$. }
\label{fig:E6pbounds}
\end{figure}

 \autoref{fig:E6pbounds} (Left) shows $\gamma(W(E_6),p,k)$ for $k=1,\ldots,17$ and various $p$.  With a small error probability of $p=0.01$ and the optimal value of $k=6$ \autoref{alg:naivealgorithm}  recovers $W(E_6)$ only $93\%$ of the time, showcasing the need for an improved recovery algorithm, even when $p$ is small. 
\autoref{fig:E6pbounds} (Right) plots the bounds given by \autoref{cor:prob_bounds_for_half_prob} on $p$ so that $\gamma(G,p,k)$ succeeds with probability at least $\frac 1 2$ using $|G| = 51840$ and $|G|<4^{27}$. Also displayed is the true bound of $\omega_{0}(G,k)$.
\end{example}

\subsection{Error Detection and \autoref{alg:naivealgorithm}}
\label{secsec:error_detected_naive}
The success rate $\gamma(G,p,k)$ of \autoref{alg:naivealgorithm}  is often lower than desired, but can be increased by either sample error detection or group error detection. Specifically, given a property $\mathcal P$ which holds for all permutations in $G$ or a property $\mathcal Q$ which holds for $G$ we have two strategies for improving this success rate: 
\begin{itemize}
\item (Sample Error Detection) Increase the success rate $\gamma(G,p,k)$ from $ (1-q)^k\varphi_k(G)$ to \\ $\gamma_{\mathcal P}(G,p,k) = (1-q_{\mathcal P})^k\varphi_k(G)$, the success rate of  \autoref{alg:naivealgorithm} on $X_{\mathcal P}$ (see \autoref{lem:naiveprob}).
\item (Group Error Detection) Increase $\gamma(G,p,k)$ to $\gamma_{\mathcal Q}(G,p,k)$, the success rate  obtained by repeating \autoref{alg:naivealgorithm} until  the result has property $\mathcal Q$ (see \autoref{alg:errordetectinggrouprecovery}).
\end{itemize}
 
\noindent The following lemma quantifies the benefit of using sample error detection. 
 \begin{lemma}
 \label{lem:reduction_factor}
 Suppose $\mathcal P$ is a property of permutations such that $G \subseteq S_{\mathcal P}$. Then 
 \[
 \gamma_{\mathcal P}(G,p,k) \geq \frac 1 2 \iff  q \leq R \cdot  \omega_0(G,k) \iff p \leq \frac{\omega_0(G,k)}{(B-A)+\omega_0(G,k)(1-B)}
 \]
 where $\mydef{R} = \frac{q}{q_{\mathcal P}}= \frac{(1-A)(1-p(1-B))}{B-A}$, $A = \frac{|G|}{n!}$, and $B= \frac{|\mathcal S_{\mathcal P}|}{n!}$.
 \end{lemma}
 \begin{proof}
 Recall that by  \autoref{lem:pmf_sample_ed}, we have that $q_{\mathcal P} = \frac{(B-A)p}{1-p(1-B)}$ and $q=p(1-A)$ and so the expression for $R$ follows algebraically. By \autoref{cor:prob_bounds_for_half_prob} we have that $\gamma_\mathcal P(G,p,k) \geq \frac 1 2 \iff q_{\mathcal P} \leq \omega_0(G,k)$. Substituting $p = \frac{q}{1-A}$ into this inequality and simplifying results in the equivalent inequality 
 \[
 q < \frac{\omega_0(G,k)(1-A)}{(B-A)+\omega_0(G,k)(1-B)}.
 \]
Dividing both sides by $(1-A)$ yields the equivalent inequality in terms of $p$.
 \end{proof}
 \begin{corollary}
 \label{cor:simpler_R}
 For $G \subset S_{\mathcal P}$ and  $p\leq \frac 1 2$, the error reduction factor $R$ is bounded below by 
 \[
 R =\frac{q}{q_{\mathcal P}} \geq \frac 1 2(B+B^{-1})
 \]
 where $B = \frac{|S_{\mathcal P}|}{n!}$. In particular, 
 $
 q_{\mathcal P} = \frac{q}{R} < \frac{2q}{B+B^{-1}} \leq \frac{2p}{B+B^{-1}} \leq \frac{2\widetilde{p}}{B+B^{-1}}
$
 whenever $p< \widetilde{p}$. 
 \end{corollary}
 \begin{proof}
 We rewrite the formula of \autoref{lem:reduction_factor} using the identity $1-A = (B-A)+(1-B)$, then apply the bounds $p(1-B) \leq p$, $B-A \leq B$, and $1-p \geq \frac 1 2 $ (when $p \leq \frac 1 2$):
\[
 R = \frac{(1-A)(1-p(1-B))}{B-A} 
  = 1+(1-B) \left(\frac{1-p(1-B)}{B-A} - p\right) \geq 1+(1-B)\left(\frac{1-p}{B} - p\right) \]
 which is at least $\geq 1+\frac{(1-B)^2}{2B}=\frac{1}{2}\left(B+B^{-1}\right)$ since  $p\leq \frac 1 2$.
 \end{proof}

Let \mydef{$\mathcal Q$} be a property of permutation groups known to hold for  $G$. \autoref{alg:naivealgorithm} can be post-processed and repeated until the output $\widehat{G}$ has property $\mathcal Q$. We call the algorithm which runs \autoref{alg:naivealgorithm} until an output has property $\mathcal Q$ the \mydef{group recovery algorithm with $\mathcal Q$-error detection}.

\begin{algorithm}[!htpb]
  \SetKwIF{If}{ElseIf}{Else}{if}{then}{elif}{else}{}%
  \DontPrintSemicolon
  \SetKwProg{Error Detecting Group Recovery}{Error Detecting Group Recovery}{}{}
  \LinesNotNumbered
  \KwIn{$(X,k,\mathcal Q)$ \\ $\bullet$ The ability to sample from the distribution of $X$ \\   $\bullet$ A sample size $k$\\ $\bullet$ A property ${\mathcal Q}$ satisfied by $G$  }
  \KwOut{$G$}
  \nl Set $\widehat{G}$ to be the output of \autoref{alg:naivealgorithm} on input $(X,k)$.\;
  \nl \textbf{while} $\widehat{G}$ does not have property $\mathcal Q$,  set $\widehat{G}$ to be the output of a repetition of \autoref{alg:naivealgorithm}.\;
  \nl \Return{$\widehat{G}$}
  \caption{Group Recovery with $\mathcal Q$-Error  Detection\label{alg:errordetectinggrouprecovery}}
\end{algorithm}
We write $\mydef{\gamma_{\mathcal Q}(G,p,k)}$ for the success rate of \autoref{alg:errordetectinggrouprecovery}. If \autoref{alg:errordetectinggrouprecovery} is run on $X_{\mathcal P}$ for some permutation error detector $\mathcal P$, we denote the corresponding success rate by $\mydef{\gamma_{\mathcal P, \mathcal Q}(G,p,k)}$.

\begin{lemma}
\label{lem:GroupErrorDetectionLemma}
The success rate of \autoref{alg:errordetectinggrouprecovery}  on input $(X,k, \mathcal Q)$  is $\gamma_{\mathcal Q}(G,p,k) = \frac{a}{a+b}$
where 
\[\gamma(G,p,k) = a=\Pr(\langle X_1,\ldots, X_k \rangle = G) \text{ and }
b =\Pr(G \neq \langle X_1,\ldots, X_k \rangle \text{ has property }\mathcal Q).\] The expected number of runs of \autoref{alg:naivealgorithm} until some output is given is $1/(a+b)$. 
\end{lemma}
\begin{proof}
This result follows the same argument as \autoref{lem:pmf_sample_ed} using the Markov chain in \autoref{fig:markov}.
\end{proof}
\begin{example}
For $G=W(E_6)$, $p=0.75$, and $k=3$, \autoref{fig:piechart} indicates the proportions of outputs of \autoref{alg:naivealgorithm} and \autoref{alg:errordetectinggrouprecovery} which were $S_{27}$, $A_{27}$, $W(E_6)$. Simply knowing that $G = W(E_6)$ is not a giant improves the success rate substantially: 
\[\gamma(W(E_6),0.75,3) \approx 1.3\% \quad \xrightarrow{\text{group error detection}} \quad \gamma_{\text{non-giant}}(W(E_6),0.75,3) \approx 87\%.\] 
\begin{figure}[!htpb]
\includegraphics[scale=0.3]{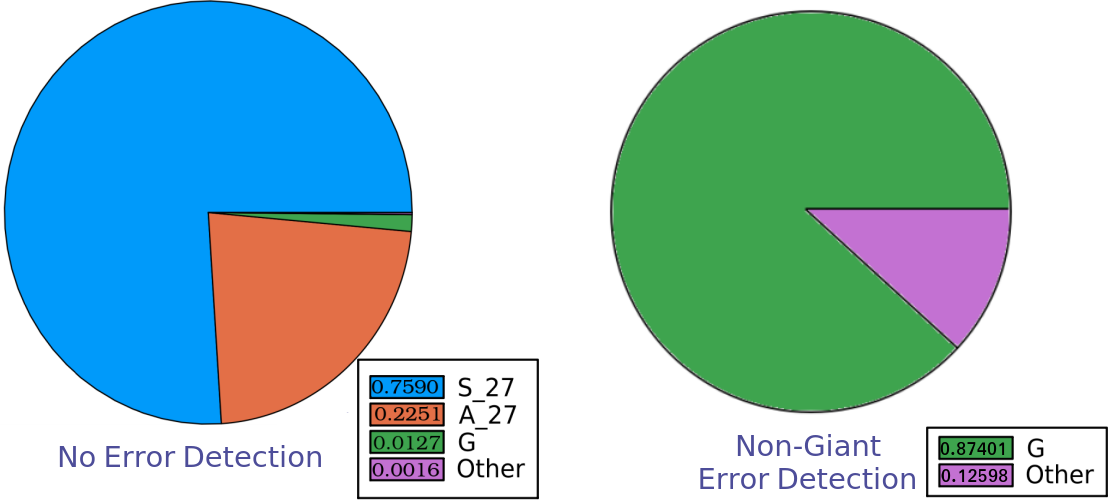}
\caption{Results of  \autoref{alg:naivealgorithm} (left) and \autoref{alg:errordetectinggrouprecovery} (right) on $10,000$ runs with  $G=W(E_6)$, $p=0.75$, and $k=3$ and $\mathcal Q\colon $ non-giant. }
\label{fig:piechart}
\end{figure}

The expected number of calls to \autoref{alg:naivealgorithm} required to obtain one output of \autoref{alg:errordetectinggrouprecovery} is $\frac{1}{0.0127+0.0016} \approx 70$. Remarkably, $G$ can be recovered frequently despite the large error rate $p=0.75$.
\end{example}

\subsection{NiAGRA: Amplifying  \autoref{alg:naivealgorithm}} 
\label{secsec:amplified_naive}
For sufficiently small $p$ and large $k$,  \autoref{alg:naivealgorithm} is far more likely to return $G$ than any other group. Thus, repeating \autoref{alg:naivealgorithm} several times and returning the mode of the results improves its success rate. This is a well-known procedure in randomized algorithms called \mydef{amplification}. For ease of analysis, we assume that the  algorithm succeeds with probability at least $\frac 1 2$. Combining \autoref{alg:naivealgorithm} with both error-detector methods, and amplification, results in the fundamental algorithm underlying our full pipeline: \textbf{NiAGRA}. 

\begin{algorithm}[!htpb]
  \SetKwIF{If}{ElseIf}{Else}{if}{then}{elif}{else}{}%
  \DontPrintSemicolon
  \SetKwProg{Naive Amplified Group Recovery Algorithm}{Naive Amplified Group Recovery Algorithm}{}{}
  \LinesNotNumbered
  \KwIn{\\ $\bullet$ The ability to sample from the distribution of $X_{\mathcal P}(G,p)$ (note: $\mathcal P$ may be trivial) \\  $\bullet$ A sample size $k$  \\
  $\bullet$ An amplification parameter $N \in \mathbb{N}$\\
  $\bullet$ A group error detector $\mathcal Q$ (note: $\mathcal Q$ may be trivial)\\
  \textbf{Assume:}  $\gamma_{\mathcal P,\mathcal Q}(G,p,k)=\frac{1}{2}+\delta$ for $\delta>0$ \\} 
  \KwOut{$G$} 
  \nl Run \autoref{alg:errordetectinggrouprecovery}   $N$ times on $(X_{\mathcal P}, k,\mathcal Q)$  to obtain $G_1,\ldots,G_N$ and \textbf{return} the mode. \;
  \caption{\textbf{N}a\textbf{i}ve \textbf{A}mplified  \textbf{G}roup \textbf{R}ecovery \textbf{A}lgorithm (\textbf{NiAGRA})\label{alg:amplifiednaivegrouprecovery}}
\end{algorithm}

\begin{lemma}
\label{lem:amplificationbounds}
Suppose $\gamma_{\mathcal P,\mathcal Q}(G,p,k)=\frac{1}{2}+\delta$ for $\delta>0$. Then for \mydef{amplification parameter} $N$, 
\[\Pr(\textbf{NiAGRA} \text{ returns }G) \geq 1-e^{-2\delta^2N}.
\] For
$N \geq \mydef{N(\alpha,\delta)} = -\frac{\log(\alpha)}{{2\delta^2}},
$
\textbf{NiAGRA} recovers $G$ with probability at least $1-\alpha$. 
\end{lemma}
\begin{proof}
Let $G_1,\ldots,G_N$ be the $N$ observations of the output of \autoref{alg:errordetectinggrouprecovery}. \textbf{NiAGRA} returns $G$ if more than $N/2$ of these groups are $G$. Write $A$ for the random variable $A_1+\cdots+A_N$ where $A_i\sim \textrm{Ber}(\gamma_{\mathcal P, \mathcal Q}(G,p,k))$ are \iid Bernoulli random variables which are $1$ if $G_i=G$ and $0$ otherwise. Then $A$ has a binomial distribution $\textrm{Bin}(\gamma_{\mathcal P, \mathcal Q}(G,p,k),N)$. The $N/2$ tail of a binomial distribution is bounded via \mydef{Hoeffding's inequality} \cite{Hoeffding}:
$
\Pr(\textbf{NiAGRA} \text{ fails})  \leq \Pr(A < N/2) \leq e^{-2\delta^2N}.
$ This upper bound on the failure rate is at most $\alpha$ when $N \geq N(\alpha,\delta)$. 
\end{proof}

\begin{example}[Amplification and the problem of $27$ lines]
\label{ex:27LinesContinued}
Continuing \autoref{ex:27lines}, \autoref{lem:amplificationbounds} states that for $(k,p) = (6,0.01)$,  \textbf{NiAGRA} recovers $W(E_6)$ more than $99.9\%$ of the time, using  $N=\lceil N(0.001,.93055-0.5) \rceil = \lceil 18.68\rceil = 19$, with no error detectors.  However, $N=7$ suffices: 
\[
\Pr(\textrm{output of }\textbf{NiAGRA} \textrm{ is not }W(E_6)) = {\footnotesize{\sum_{i=0}^{4} {{7}\choose{i}}\gamma(G,p,k)^i(1-\gamma(G,p,k))^{7-i} \approx 0.000685}}\]

\begin{figure}[!htpb]
\includegraphics[scale=0.48]{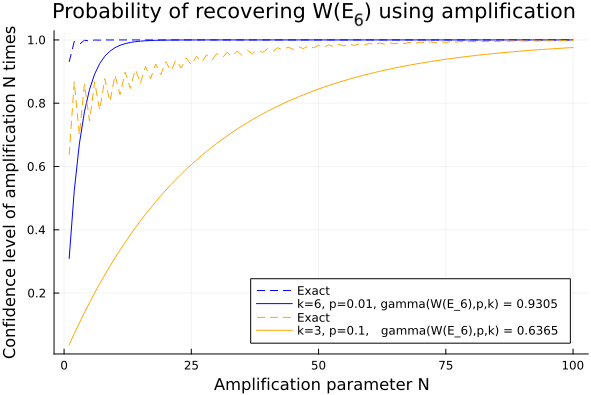}
\caption{True success rates of \textbf{NiAGRA} on $G=W(E_6)$ versus bounds of \autoref{lem:amplificationbounds}. }
\label{fig:amplifiedE6}
\end{figure}

 \autoref{fig:amplifiedE6} shows the lower bound on the probability of success of amplification on $G=W(E_6)$ with $p=0.01$ and the \textit{optimal} choice of $k=6$, along with the analogous chart for $p=0.1$ and $k=3$. We plot these values along with the true probability  that the mode of $N$ runs of  \autoref{alg:naivealgorithm} is $G$.

Requiring that $\gamma_{\mathcal P, \mathcal Q}(G,p,k) \geq 1/2$ is a conservative sufficient requirement for amplification to succeed. Indeed, as long as $G$ appears \textit{more often} than any other group, amplification works. This difference depends intimately on the structure of $G$ and we do not analyze it. Rather, we showcase the empirical values of amplification success in  \autoref{fig:heatmap} on $G=W(E_6)$. The first row of each heatmap indicates that $W(E_6)$ cannot be generated by a single element. The first column of each heatmap reflects the empirical values of $\gamma(W(E_6),p,k)$. When $\gamma(W(E_6),p,k)$ is larger than the probability that any other subgroup of $S_{27}$ is returned by \autoref{alg:naivealgorithm}, the corresponding row is increasing with respect to $N$. The last $(p,k)$ value for which this occurs is $(p,k)=(0.33,2)$.
\end{example}

\begin{figure}[!htpb]
\includegraphics[scale=0.42]{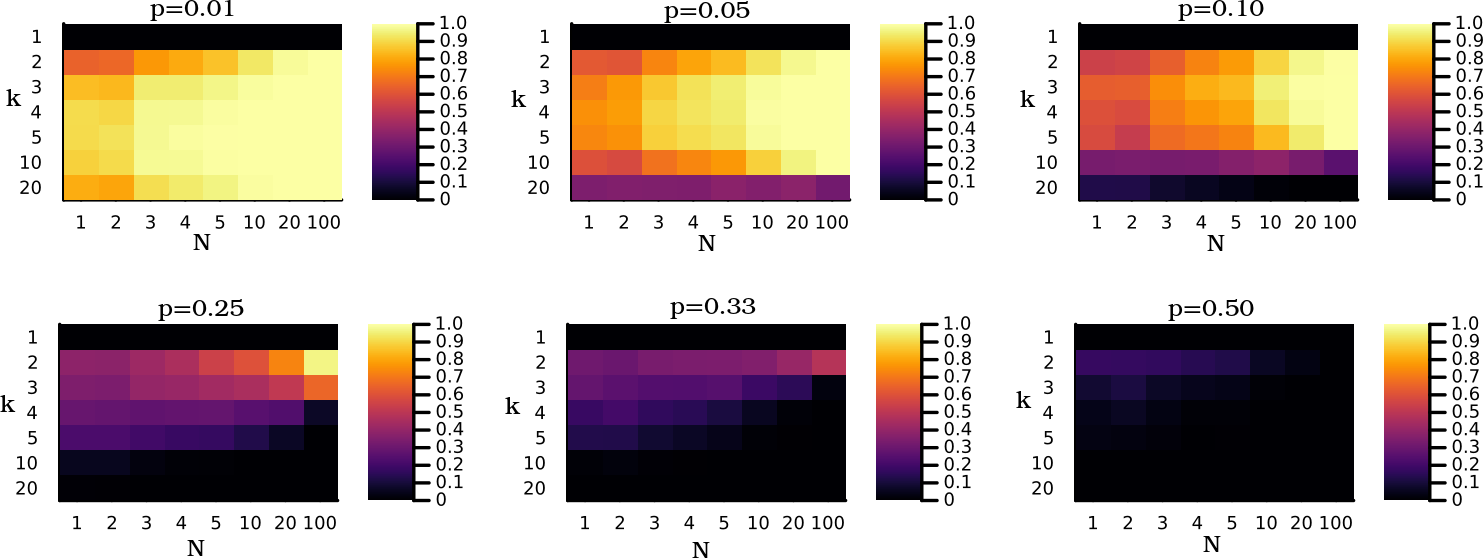}
\caption{Heatmaps of success rates of \textbf{NiAGRA} on $G=W(E_6)$ as $p$, $k$, and $N$ vary. }
\label{fig:heatmap}
\end{figure}

Although we do not have access to the value $\gamma_{\mathcal P, \mathcal Q}(G,p,k)$, we often have access to bounds $\widetilde{p}$, $\widetilde{M}$, and $\widetilde{B}$ on  $p$, $\lceil\log_2(|G|)\rceil$, and $\frac{|S_{\mathcal P}|}{n!}$, which we may use to bound 
$\gamma_{\mathcal P, \mathcal Q}(G,p,k)$:
\begin{align*}\gamma_{\mathcal P, \mathcal Q}(G,p,k) \geq \gamma_{\mathcal P}(G,p,k) &\geq (1-q_{\mathcal P})^k \left(1-\frac{8}{2^{k-\lceil \log_2(|G|)\rceil}} \right) \\ &\geq \left(1-\frac{2\widetilde{p}}{B+B^{-1}}\right)^k\left(1-\frac{8}{2^{k-\widetilde{M}}}\right).
\end{align*}
Checking if the right-hand-side is at least $\frac{1}{2}$ is called a \mydef{success rate check}.

\begin{algorithm}[]
  \SetKwIF{If}{ElseIf}{Else}{if}{then}{elif}{else}{}%
  \DontPrintSemicolon
  \SetKwProg{Success Rate Check}{Success Rate Check}{}{}
  \LinesNotNumbered
  \KwIn{$(\widetilde{p},\widetilde{M},\widetilde{B},k)$ \\ $\bullet$ An upper bound $\widetilde{p}$ on the error rate $p$ \\  $\bullet$ An upper bound $\widetilde{M}$ on $\lceil \log_2(|G|) \rceil$ \\ $\bullet$ an upper bound on the proportion $\frac{|S_{\mathcal P}|}{n!}$ of permutations in $S_n$ which satisfy $\mathcal P$. \\ $\bullet$ the number $k$ of samples used in \autoref{alg:naivealgorithm}  }
  \KwOut{A lower bound on $\gamma_{\mathcal P, \mathcal Q}(G,p,k)$.} 
  \nl \textbf{return} $\left(1-\frac{2\widetilde{p}}{B+B^{-1}}\right)^k\left(1-\frac{8}{2^{k-\widetilde{M}}}\right)$
  \caption{Success Rate Check\label{alg:successratecheck}}
\end{algorithm}

\section{Property Recovery}
\label{sec:property_recovery}
In this section, we list several tests for extracting information about $G$ from $X=X(G,p)$ or some $X_{\mathcal P}$. Each test determines a property, $\mathcal P$ or $\mathcal Q$, for sample error detection or group error detection, respectively. Additional information, like bounds for $p$, $|G|$, or the proportion $B$ of permutations which satisfy $\mathcal P$, can sometimes be obtained as well.

\subsection{The Giant Test} \label{secsec:giant_test} The \mydef{giant test} is a two-sided randomized  algorithm, or hypothesis test, for determining if $G$ is a giant. It relies on approximating the expected value of the random variable \[
\mydef{\mathcal G(X_1,X_2)} = \langle X_1,X_2 \rangle \text{ is a giant}.
\] The expected value of this Bernoulli random variable differs greatly depending on whether $G$ is a giant or not.  The main technical result of this subsection pinpoints when these probabilities separate, enabling us to use basic statistical hypothesis testing to establish if $G$ is a giant.

\begin{theorem}
\label{thm:giant_bounds}
Let $G \leq S_n$ and let $\ell(n)$ and $u(n)$ be the  bounds of \autoref{prop:morganroneydougal}.
If $G$ is a giant, then the probability that $\langle X_1,X_2 \rangle$ is a giant is bounded by
\[
\mydef{\mathcal L(n,p)} = (1-p+p^2)\ell(n) <\Pr(\mathcal G(X_1,X_2)) < (1-p+p^2)u(n)+p(1-p).
\]
If $G$ is not a giant, then the probability that $\langle X_1,X_2 \rangle$ is a giant is bounded by 
\[
p^2 \ell(n) \leq \Pr(\mathcal G(X_1,X_2)) \leq 2p(1-p)+p^2u(n) = \mydef{\mathcal U(n,p)}.
\]
The quantity $\mathcal U(n,p)$ is less than $\mathcal L(n,p)$ whenever 
\[
p<\mydef{b_n} = \frac{{3n^2-n-8.8}-\sqrt{n^4+2n^3+50.08n^2-13.88n-199.584}}{4n^2-15.74} = 0.5- 0.5 \frac 1 n - 6.3675 \frac{1}{n^2}+O(n^{-3})
\]
which begins $b_5=0.214,b_{10}=0.395, b_{50} = 0.487$ and $\lim_{n \to \infty} b_n = 0.5$.
\end{theorem}
\begin{proof} If $G$ is a giant then
\[
\Pr(\mathcal G(X_1,X_2)) = (1-p)^2\Pr(\mathcal G(Y_1,Y_2))+2p(1-p)\Pr(\mathcal G(Y_1,Z_1))+ p^2\Pr(\mathcal G(Z_1,Z_2)) 
\]
The probabilities $\Pr(\mathcal G(Y_1,Y_2))$ and $\Pr(\mathcal G(Z_1,Z_2))$ are both bounded below by $\ell(n)$ and above by $u(n)$. The remaining probability $\textrm{Pr}(\mathcal G(Y_1,Z_1))$ is similarly bounded  if $G = S_n$. If $G = A_n$ however, 
\begin{align*}
\Pr(\mathcal G(Y_1,Z_1)) &= \Pr(Z \in A_n)\cdot \Pr(\mathcal G(Y_1,Y_2))+ \Pr(Z \not\in A_n) \cdot \Pr(\mathcal G( U_{A_n},U_{S_n-A_n})) \\
&= \frac 1 2 \Pr(\mathcal G(Y_1,Y_2))+ \frac 1 2 \cdot \Pr(\mathcal G( U_{A_n},U_{S_n-A_n}))
\end{align*}
Bounding the second summand within the interval $[0,1/2]$ we obtain the lower bound
\[
\Pr(\mathcal G(X_1,X_2)) > (1-p)^2\ell(n)+2p(1-p)\frac 1 2 \ell(n) + p^2 \ell(n) = (p^2-p+1)\ell(n)
\]
and upper bound
\[(1-p)^2u(n)+2p(1-p)\left(\frac 1 2 u(n)+\frac 1 2\right) + p^2 u(n) = (p^2-p+1)u(n)+p(1-p).
\]
If $G$ is not a giant, then 
\[\Pr(\mathcal G(Y_1,Y_2))=0 \quad \quad \Pr(\mathcal G(Y_1,Z_1)) \in (0,1), \text{ and }\quad \quad \Pr(\mathcal G(Z_1,Z_2)) \in (\ell(n),u(n)).\] Applying a similar decomposition as before to the mixture model $X$ we obtain the bounds 
\[p^2 \ell(n) \leq \Pr(\mathcal G(X_1,X_2)) \leq 2p(1-p)+p^2u(n).\]
Setting the former lower bound equal to the latter upper bound gives a quadratic formula in $p$, for which, $b_n$ is the solution. 
\end{proof}

\autoref{thm:giant_bounds} enables us to distinguish whether $G$ is a giant by estimating the expected value of the Bernoulli random variable $\mathcal G(X_1,X_2)$. By doing so, we may distinguish if $\mathbb{E}(\mathcal G(X_1,X_2))>\mathcal L(n,p)$ or $\mathbb{E}(\mathcal G(X_1,X_2))<\mathcal U(n,p)$. Since the value of $p$ is unknown \emph{a priori}, we must choose some bound $\tilde{p}$ for which $p< \tilde{p}$  and conservatively decide whether $\mathbb{E}(\mathcal G(X_1,X_2))>\mathcal L(n,\tilde{p})$ or $\mathbb{E}(\mathcal G(X_1,X_2)) < \mathcal U(n,\tilde{p})$. There are many statistical methods for performing this task. For example, one may choose sufficiently many samples $N$ to approximate the expected value to a desired $\delta$-radius. In practice, it makes sense to use the \emph{sequential probability ratio test} \cite{SPRT}. To keep things theoretically simple, we use a fixed sample size hypothesis test and obtain bounds on our confidence level using  \emph{Hoeffding's inequality}, which we already saw in \autoref{lem:amplificationbounds}. We refer to it as the \textit{mean-threshold distinguisher}.

\begin{lemma}[Hoeffding]
\label{lem:hoeffding_distinguisher}
Suppose $X$ is a random variable taking values in $[0,1]$, and let $0 \le a < b \le 1$
such that either $\mathbb{E}(X)\le a$ or $\mathbb{E}(X)\ge b$.
Let $X_1,\dots,X_N$ be i.i.d. copies of $X$, define
\[
\widetilde{X} = X_1+\cdots+X_N, \qquad
c = \frac{a+b}{2}, \qquad
\delta = \frac{b-a}{2}.
\]
If $\mathbb{E}(X)\ge b$, then
\[
\Pr(\widetilde{X} < cN)
=
\Pr\left(\mathbb{E}(\widetilde{X}) - \widetilde{X}
\ge (\mathbb{E}(X)-c)N\right)
\le
\Pr\left(\mathbb{E}(\widetilde{X}) - \widetilde{X}
\ge \delta N\right)
\le
e^{-2\delta^2 N}.
\]
Similarly, if $\mathbb{E}(X)\le a$, then
$
\Pr(\widetilde{X} > cN)
\le
e^{-2\delta^2 N}.
$
In particular, the \mydef{mean-threshold distinguisher}
\[
\mydef{\mathrm{Distinguisher}_{N,c}}(X)
=
\begin{cases}
``\mathbb{E}(X)\le a" & \widetilde{X} \le cN, \\[4pt]
``\mathbb{E}(X)\ge b" & \widetilde{X} \ge cN
\end{cases}
\]
is incorrect with probability at most $e^{-2\delta^2 N}$.
This error probability is bounded above by $\alpha$ whenever
\[
N \ge N(\alpha,\delta)
= -\frac{\log(\alpha)}{2\delta^2}
= -\frac{2\log(\alpha)}{(b-a)^2}.
\]
\end{lemma}

\begin{proof}
The statement bounding the probability that $\widetilde{X}$ deviates from its expected value by more than $\delta N$ is {Hoeffding's inequality} \cite{Hoeffding}. The bound $N(\alpha,\delta)$ is obtained by solving for $N$ in $e^{-2\delta^2N} \leq \alpha$. 
\end{proof}

\begin{algorithm}[!htpb]
  \SetKwIF{If}{ElseIf}{Else}{if}{then}{elif}{else}{}%
  \DontPrintSemicolon
  \SetKwProg{Giant Test}{Giant Test}{}{}
  \LinesNotNumbered
  \KwIn{$(X,N)$ \\ $\bullet$ The ability to sample from the distribution of $X$ \\ $\bullet$ A number $N$ of Bernoulli trials \\ }
  \KwOut{$\texttt{true}$ if $G$ is a giant and $\texttt{false}$ otherwise}
  \nl Set $c = \frac{\mathcal L(n,\widetilde{p})+\mathcal U(n,\widetilde{p})}{2}$ for any $p<\widetilde{p}<b_n$.\;
  \nl  Distinguish if $\mathbb{E}(\mathcal G(X_1,X_2))>c$ using the mean-threshold distinguisher.\;
  \nl \textbf{return} \texttt{true} if $\mathbb{E}(\mathcal G(X_1,X_2))>c$ and \texttt{false} otherwise.
  \caption{Giant Test\label{alg:gianttest}}
\end{algorithm}

\begin{theorem}
\autoref{alg:gianttest} determines correctly whether $G$ is a giant with probability at least $1-\alpha$ as long as $p<\widetilde{p}<b_n$ and $N> N\left(\alpha,\frac{\mathcal L(n,\widetilde{p})-\mathcal U(n,\widetilde{p})}{2}\right)$.
\end{theorem}
\begin{proof}
This directly follows from \autoref{thm:giant_bounds}, which guarantees that the bounds on the expected value of $\mathcal G(X_1,X_2)$ separate when $p<b_n$, along with \autoref{lem:hoeffding_distinguisher}.
\end{proof}

\begin{example}\label{ex:giantweyl} We continue our running example of $G=W(E_6) \leq S_{27}$ under the assumption that 
 $p<\frac 1 3= \tilde{p}$. Since $\frac 1 3 < b_{27} \approx 0.47$,  \autoref{thm:giant_bounds} states that 
\[\mathcal U(n,p)<\mathcal U(n,\tilde{p}) \approx 0.5512 < 0.7395 \approx \mathcal L(n,\tilde{p}) < \mathcal L(n,p).\]
See \autoref{fig:giant_means} for a visualization of these bounds.
We thus use  $c = \frac{0.5512+0.7395}{2} = 0.6454$ as a distinguishing threshold for $\mathcal G(X_1,X_2)$ applied to any permutation group in $S_{27}$. The value of $\delta$ is $0.09414$ and so to obtain a $99\%$ confidence level on the output of \autoref{alg:gianttest}, we may sample $N = \lceil N(0.01,0.09414)\rceil = 260$ observations of $\mathcal G(X_1,X_2)$.

\begin{figure}[!htpb]
\includegraphics[scale=0.18]{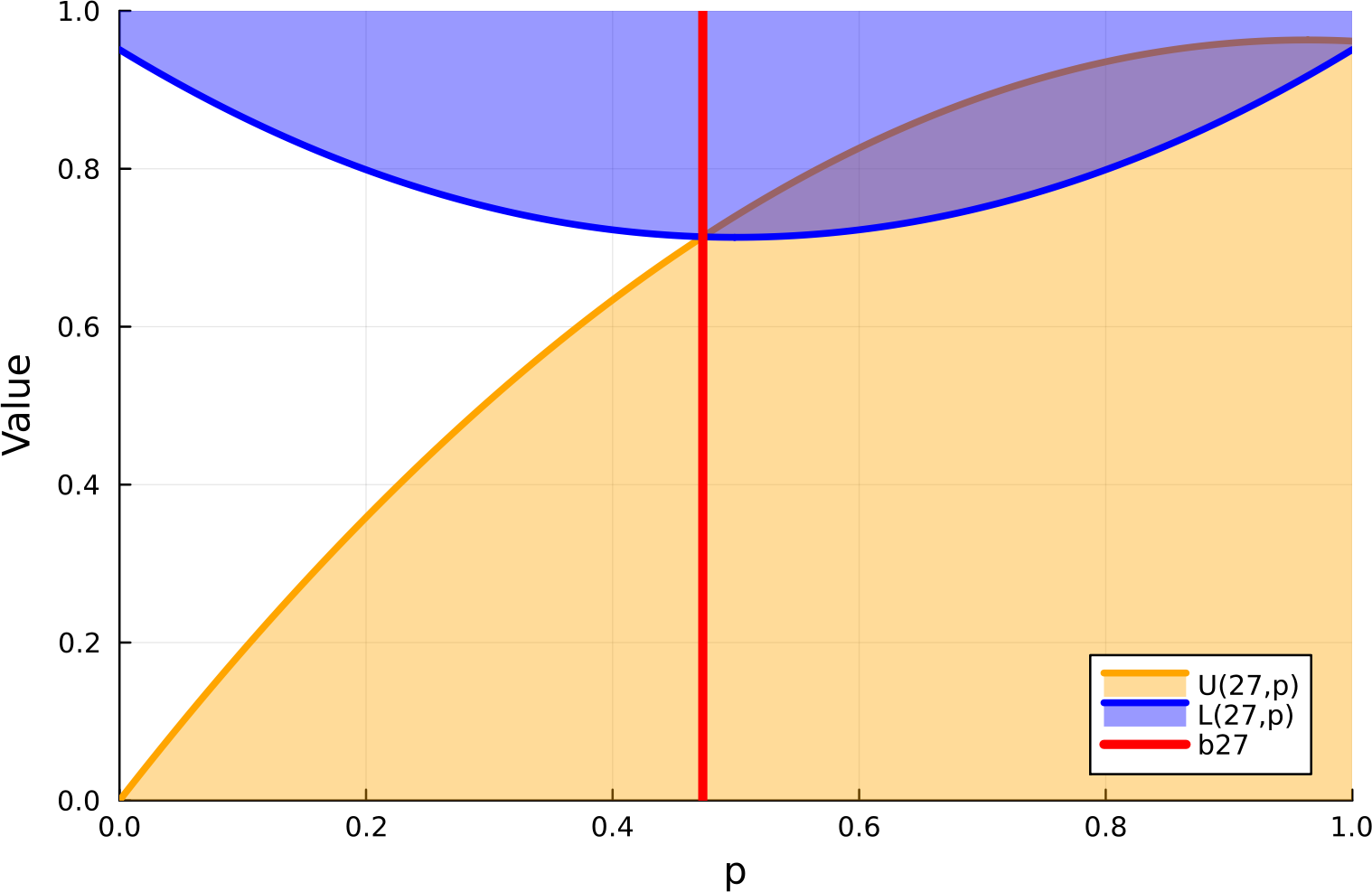}
\caption{A graph of the bound $\mathcal L(27,p)$ of the expected value of $\mathcal G(X_1,X_2)$ if $G$ is a giant (blue) and the bound $\mathcal U(27,p)$ of the expected value if $G$ is not a giant. They seperate for $p<b_{27}$.} 
\label{fig:giant_means}
\end{figure}

Running this algorithm on $p=\tilde{p} = \frac 1 3$ gives a proportion of $143/260 \approx .55$ of pairs which generate a giant, which is lower than the threshold $c$ and so \autoref{alg:gianttest} decides correctly that $G$ is not a giant. With confidence level $99\%$.

 Using $G=A_{27}$, we find that $259/260$ pairs generate a giant. Using $G=S_{27}$, this proportion became $260/260$. 
In terms of the computational cost of this test, we remark that  deciding whether $10000$ random pairs of permutations from $W(E_6)$ and $S_{27}$ generate a giant takes about $3.4$ and $3.5$ seconds, respectively,  in \texttt{GAP} \cite{GAP}.
\end{example}

\subsection{Subgroup Test} \label{secsec:subgroup_test}
The \mydef{subgroup test} determines if $G\leq S_{n}$ is a subgroup of some known group $H\leq S_n$  using only observations of $X$. 
Like the giant test, it works by estimating the expected value of a Bernoulli random variable:
\[\mydef{X \in H} = \begin{cases}
1 & X \in H \\
0 & X \not\in H
\end{cases}.
\] As in the giant test, the expected values of this  random variable separate depending on if $G \subseteq H$.

\begin{lemma}
\label{lem:subgroup_bound}
If $G \leq S_n$  and $H$ is a proper subgroup of $S_n$ then
\[
\mathbb{E}(X \in H) = \frac{(1-p)|G \cap H|}{|G|}+\frac{p|H|}{n!} = \begin{cases} 
(1-p)+\frac{p|H|}{n!} & G\leq H \\
\leq \frac{1-p}{2}+\frac{p|H|}{n!} \leq \frac{1}{2} & \text{otherwise}\\
\end{cases}
\]
In particular, these expected values differ by at least $\frac{1-p}{2}$.
\end{lemma}
\begin{proof}
The probability that an observation of $U_G$ belongs to $H$ is $[G:G \cap H]^{-1}$ which is $1$ if $G \leq H$ and at most $\frac{1}{2}$ otherwise thanks to LaGrange's theorem. In particular,
\begin{align*}
\Pr(X \in H) &= (1-p)\Pr(Y \in H) + p\Pr(Z \in H)\\  &= (1-p)[G:G \cap H]^{-1} + p[S_n:H]^{-1}
\end{align*}
which is $(1-p)+\frac{p|H|}{n!}$ if $G\leq H$. Otherwise, the first summand  is at most $(1-p)/2$ by LaGrange's theorem, and similarly, the second summand is at most $p/2$.
\end{proof}

\begin{algorithm}[!htpb]
  \SetKwIF{If}{ElseIf}{Else}{if}{then}{elif}{else}{}%
  \DontPrintSemicolon
  \SetKwProg{Subgroup Test}{Subgroup Test}{}{}
  \LinesNotNumbered
  \KwIn{$(X,N,H)$ \\ $\bullet$ The ability to sample from the distribution of $X$ \\ $\bullet$ $N$ a number of samples \\ $\bullet$ A proper subgroup $H \leq S_n$ \\   }
  \KwOut{\texttt{true} if $G \leq H$ and \texttt{false} otherwise}
  \nl Set  $c = \frac 3 4 -\frac{\widetilde{p}(n!-|H|)}{2n!} $ for any $\widetilde{p}>p$.\;
  \nl Distinguish if $\mathbb{E}(X \in H)>c$  using the mean-threshold distinguisher on $N$ samples. \;
  \nl \textbf{return} \texttt{true} if $\mathbb{E}(X \in H) > c$ and \texttt{false} otherwise
  \caption{Subgroup Test\label{alg:subgrouptest}}
\end{algorithm}
\begin{theorem}\label{thm:subgroup_test} For any $G \leq S_n$ and $H \lneq S_n$, 
\autoref{alg:subgrouptest} correctly determines if $G \leq H$ with confidence $1-\alpha$ when $p<\widetilde{p}\leq \frac 1 2$ and $N > N\left(\alpha,\frac{1-2\widetilde{p}}{ 4}\right)$.
\end{theorem} 
\begin{proof}
This follows from \autoref{lem:subgroup_bound} and \autoref{lem:hoeffding_distinguisher}, since the relevant $\delta$ is $c-\frac 1 2 = \frac 1 4 -\frac{\widetilde{p}(n!-|H|)}{2n!}$ which is greater than $\frac{1-2\widetilde{p}}{ 4}$.
\end{proof}

\noindent The subgroup test is versatile and can be specialized to the following settings.
\begin{itemize}
\item (The Alternating Test) The \mydef{alternating test} is the application of the subgroup test when $H=A_n$. This is most useful when $G$ has already been determined to be a giant by the giant test, in which case a further distinction between $A_n$ and $S_n$ identifies $G$. 
\item (The Block Test) The \mydef{block test} is the application of the subgroup test when $H$ is a wreath product $\textrm{Wr}(\mathcal B)$ for some partition $\mathcal B$ of $[n]$ into blocks of equal sizes. This test determines if $\mathcal B$ is a block structure of a transitive group $G$.
\item (The Orbit Refining Test) The \mydef{orbit refining test} is the application of the subgroup test when $H$ is a Young subgroup $S_{\Delta}$ corresponding to some partition $\Delta$ of $[n]$. The orbit test determines whether the orbit partition  of $G$ refines the partition $\Delta$. This is an important subroutine of \autoref{alg:OrbitConfirmation}.
\end{itemize}

\begin{example}
We illustrate the subgroup test (specifically, the alternating test) by verifying that $W(E_6) \leq A_{27}=H$ using $X =X(W(E_6),\frac{1}{3})$ under the assumption $p<\widetilde{p}=1/2$. We compute 
$\delta = \frac{1}{4}-\frac{\widetilde{p}}{2}+\frac{\widetilde{p}}{4} = 0.125$.
and $N(0.01,\delta) = 147.36$. The sample mean of $N=148$ i.i.d. observations of $X =X(W(E_6),\frac{1}{3}) \in A_{27}$ w is $\overline{\mu} = 0.83783$.
 This is well-above the distinguishing threshold $c=\frac{3(1-1/2)}{4} + (1/2)(1/2) = 0.625$. Thus, the mean-threshold distinguisher declares that  $G \leq A_{27}$. 
\end{example}

Equipped with the subgroup test, we adapt the the naive group recovery algorithm (\autoref{alg:naivealgorithm}) to find supergroups of $G$ using only $X$ and the subgroup test (\autoref{alg:subgrouptest}).
\begin{algorithm}[!htpb]
  \SetKwIF{If}{ElseIf}{Else}{if}{then}{elif}{else}{}%
  \DontPrintSemicolon
  \SetKwProg{Find Supergroup}{Find Supergroup}{}{}
  \LinesNotNumbered
  \KwIn{$X$ \\ $\bullet$ The ability to sample from the distribution of $X(G,p)$  \\ $\bullet$ (Optional) Any property $\overline{\mathcal Q}$ which  is true for every supergroup of $G$ }
  \KwOut{A supergroup of $G$ satisfying $\overline{\mathcal Q}$}
  \nl Set $\texttt{gens} = \{\}$\;
  \nl Set $\texttt{gens} = \texttt{gens} \cup \{\sigma\}$ for $\sigma \sim X$. \;
  \nl Set $H=\langle g \mid g \in \texttt{gens} \rangle$\;
  \nl Check if $H$ satisfies $\overline{\mathcal Q}$. If not, go to $\texttt{2}$.\;
  \nl Check if $G \leq H$ using the subgroup test. If not, go to $\texttt{2}$.\;
  \nl \textbf{return} $H$
  \caption{Find Supergroup\label{alg:findsupergroup}}
\end{algorithm}

The main utility of \autoref{alg:findsupergroup} is that it may be used to produce new properties $\mathcal P$ satisfied by elements of $G$, namely $\mathcal P\colon \sigma \in H$. There is a positive probability that the returned group is $G$ itself, in which case this error detection method reduces $q_{\mathcal P}$ to $0$. One may notice that in our main algorithm, illustrated in \autoref{fig:flowchart}, finding supergroups and using the corresponding error detection property is the final loop after all other algorithms described in this paper have been applied. 

\begin{example}
\label{ex:27linesSupergroups}
Let $G = W(E_6)$ and take $p=0.2$, $\widetilde{p}=0.33$, and $\alpha =0.2$. \autoref{fig:supergroup_pie} shows two pie charts indicating the results of $N=1000$ applications of \autoref{alg:findsupergroup}, using the subgroup test with error bound $\widetilde{p}$ and confidence level $0.8$. The first pie chart uses no property $\overline{\mathcal Q}$ but the second pie chart uses the property that every supergroup of $G$ is transitive (since $G$ is).

\begin{figure}[!htpb]
\includegraphics[scale=0.4]{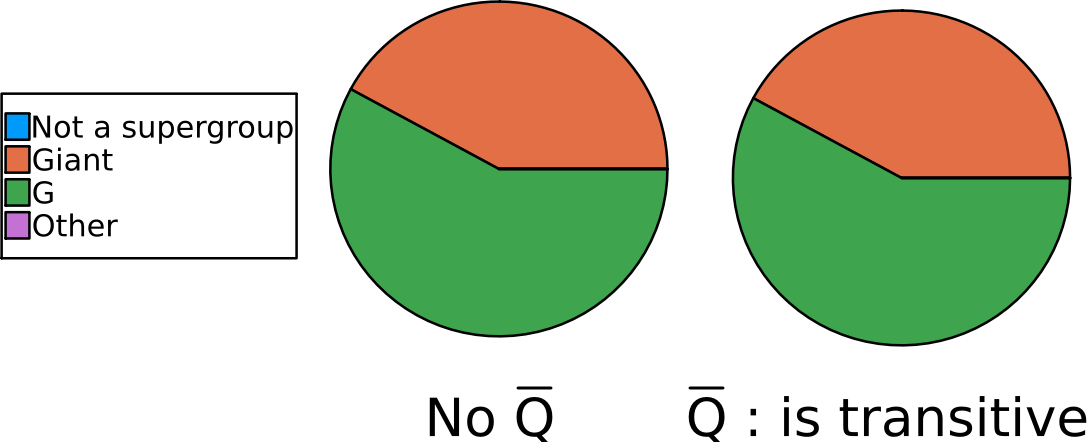}
\caption{The results of \autoref{alg:findsupergroup} ($N=10,000$) using no $\overline{\mathcal Q}$ and  $\overline{\mathcal Q}\colon$is transitive. }
\label{fig:supergroup_pie}
\end{figure}

 Notably, the pie charts are essentially identical. This is because equipping \autoref{alg:findsupergroup} with transitivity information only eliminates intransitive results, a rare potential output. The benefit, however, is shown in  \autoref{fig:samplehistogram}: the number of samples of $X$ necessary to verify a supergroup of $G$ is significantly reduced since recognizing intransitivity bypasses subgroup tests.
\begin{figure}[!htpb]
\includegraphics[scale=0.2]{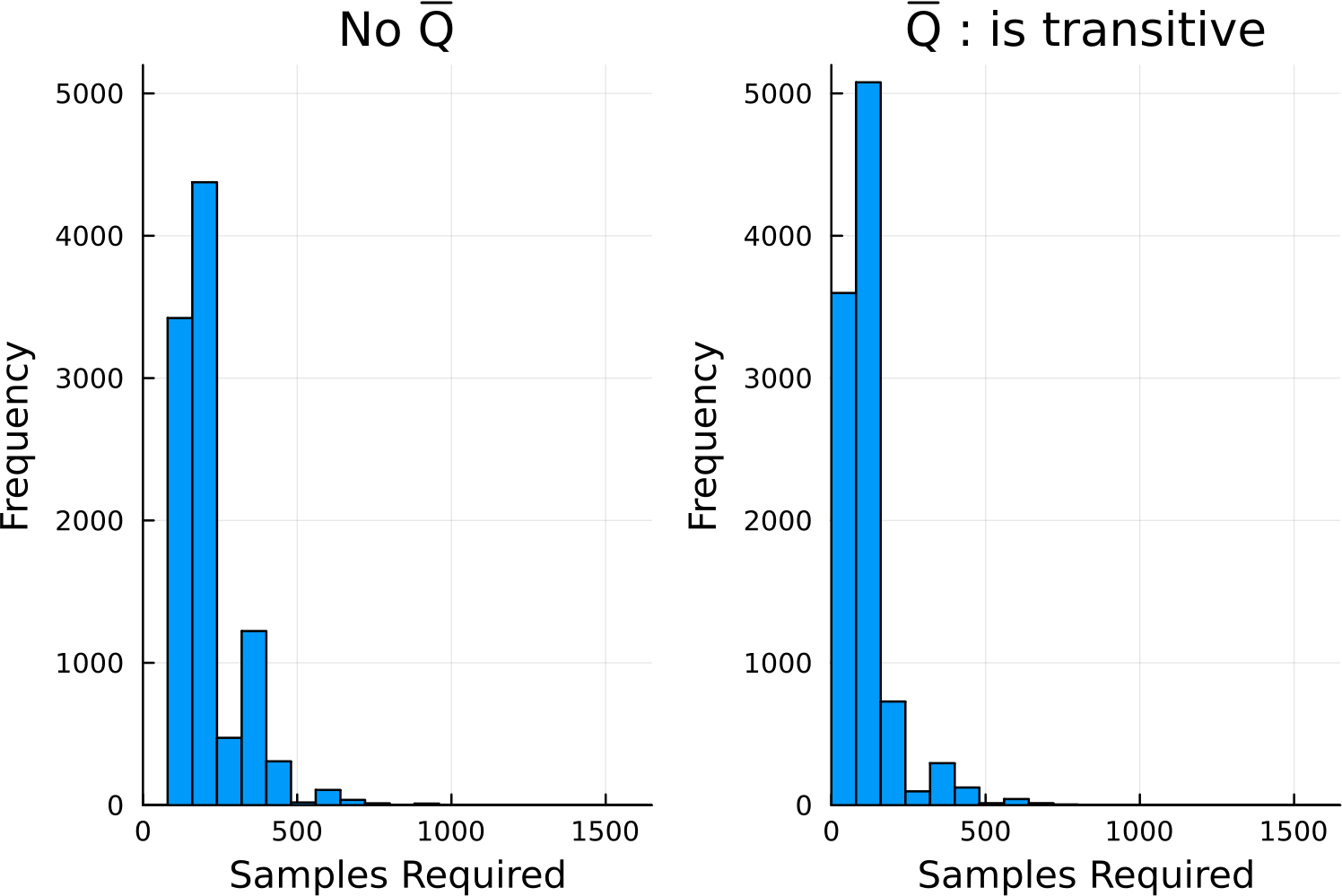}
\caption{Two histograms indicating the distribution of samples required to run \autoref{alg:findsupergroup} in the cases that the user passes (left)   no $\overline{\mathcal Q}$ and (right) $\overline{\mathcal Q}\colon$is transitive.}
\label{fig:samplehistogram}
\end{figure}
\end{example}

\subsection{Transitivity test} \label{secsec:transitivity_test} The \mydef{transitivity test} determines if $G$ is transitive by estimating the expected value of $\textrm{Fix}(X)$ and applying Burnside's lemma (\autoref{lem:Burnside}).
\begin{lemma}
\label{lem:BurnsideSeparator}
Let $G \leq S_n$ be a permutation group with $m$ orbits, then 
\[
\textrm{Fix}(X) = (1-p)\textrm{Fix}(Y)+p\textrm{Fix}(Z) \quad \quad \text{ and } \quad \quad \mathbb{E}(\textrm{Fix}(X)) = (1-p)m +p.
\]
\end{lemma}
\begin{proof}
This follows directly from Burnside's lemma and linearity of expectation.
\end{proof}
\autoref{lem:BurnsideSeparator} can be easily extended to the random variable $\textrm{Fix}_k(X)$  which counts the number of distinct ordered $k$-tuples which are fixed by an observation of $X$:
\[
\mydef{\textrm{Fix}_k(X)} = \begin{cases}
0 & \textrm{Fix}(X)<k \\
\frac{\textrm{Fix}(X)!}{(\textrm{Fix}(X)-k)!} & \text{ otherwise}.
\end{cases}\]
\begin{corollary}
Suppose the action of $G \leq S_n$  on distinct ordered $k$-tuples has $m$ orbits. Then,
\[
\textrm{Fix}_k(X) = (1-p)\textrm{Fix}_k(Y)+p\textrm{Fix}_k(Z) \quad \quad \text{ and } \quad \quad \mathbb{E}(\textrm{Fix}_k(X)) = (1-p)m +p.
\]
\end{corollary}

\begin{algorithm}[!htpb]
  \SetKwIF{If}{ElseIf}{Else}{if}{then}{elif}{else}{}%
  \DontPrintSemicolon
  \SetKwProg{$k$-Transitivity Test}{$k$-Transitivity Test}{}{}
  \LinesNotNumbered
  \KwIn{$(X,N)$ \\ $\bullet$ The ability to sample from the distribution of $X$ \\ $\bullet$ $N$ a number of samples \\$\bullet$  $k \in [n]$ \\  }
  \KwOut{\texttt{true} if $G$ is at least  $k$-transitive and \texttt{false} otherwise}
  \nl Let $c = \frac{3-\widetilde{p}}{2}$ for any $\widetilde{p}>p$.\;
  \nl Distinguish if $\mathbb{E}(\textrm{Fix}_k(X))<c$ using the mean-threshold distinguisher on $N$ samples. \;
  \nl \textbf{return} \texttt{true} if $\mathbb{E}(\textrm{Fix}_k(X))<c$ and \texttt{false} otherwise.
  \caption{$k$-Transitivity Test\label{alg:transitivitytest}}
\end{algorithm}

\begin{theorem}
\label{thm:transitivity_test_thm}
\autoref{alg:transitivitytest} determines whether $G$ is $k$-transitive with probability at least $1-\alpha$ when $p<\widetilde{p}$ and $N>N\left(\alpha,\frac{1-\widetilde{p}}{2n\cdot(n-1)\cdots(n-k+1)}\right)$. When $k=1$, this specializes to $N>N\left(\alpha,\frac{1-\widetilde{p}}{2n}\right)$
\end{theorem}
\begin{proof}
As usual, this is a combination of the separation bounds of \autoref{lem:BurnsideSeparator} and the elementary mean-threshold distinguisher of \autoref{lem:hoeffding_distinguisher}. Since $\textrm{Fix}_k(X)$ takes values bounded between $0$ and the falling factorial $n\cdot(n-1)\cdots(n-k+1)$, we normalize and use the random variable $\frac{\textrm{Fix}_k(X)}{n\cdot(n-1)\cdots(n-k+1)}$ which is $[0,1]$-valued. The result then follows from Hoeffding's inequality.
\end{proof}

\begin{remark}[$k$-transitivity is hard]
\label{rem:k_transitivity_hard}
We include the $k$-transitivity test as an extension of the transitivity test for the sake of completion. Knowing that $1$-transitivity fails, or that $2$-transitivity holds, earns information about $G$ that can be used for error detection. If $G$ is not transitive, then there is a Young supergroup corresponding to the orbits which can be used for sample error detection. If $G$ is $2$-transitive then $G$ is primitive and $\mathcal Q\colon$\textit{is primitive} may be used for group error detection. Beyond $k=2$, we do not see this as an effective test as the sample sizes required for high confidence become prohibitively large. 
\end{remark}

\subsection{Orbit Agreement, Confirmation, and Recovery} \label{secsec:orbit_recovery} Suppose that the orbits $\Delta=(\Delta_1,\ldots,\Delta_m)$ of $G \leq S_n$ have sizes $n_1,\ldots,n_m$ respectively. 
 Write $\mydef{i \sim j}$ if $i$ and $j$ belong to the same orbit of $G$ and $\mydef{G.i}$ for the orbit of $G$ containing $i$. The key to recovering the orbits of $G$ is understanding the random variable
\[
\mydef{X^{i \mapsto j}} = [X(i)=j] =\begin{cases} 1 & X(i)=j \\ 0 & \text{otherwise} \end{cases}.
\] 
We first analyze $Y^{i \mapsto j}$ when $G$ is transitive.
 \begin{lemma}
\label{lem:prob_of_moving}
If $G \leq S_n$ is a transitive group and $i \sim j$ then $\mathbb{E}(Y^{i \mapsto j}) = \Pr(Y(i)=j)=\frac{1}{n}$.
\end{lemma}
\begin{proof}
Burnside's lemma states that the expected number of fixed points in a transitive group is $1$. Let $\tau$ be an element of $G$ which maps $i$ to $j$. Every point has an equal probability of being fixed since $\sigma \in G$ fixes $i$ if and only if its image $\tau\sigma \tau^{-1}$ under the inner automorphism of conjugation by $\tau$ fixes $j$. Similarly, the probability that $i$ maps to $i$ is the same as the probability that $i$ maps to $j$ since $\sigma(i)=i \iff \tau\sigma(i)=j$ and $\sigma \mapsto \tau \sigma$ is a bijection. 
\end{proof}

\begin{lemma}
\label{lem:uniform_homomorphism}
Let $\psi: G \to H$ be a group homomorphism. If $Y$ is distributed uniformly on $G$ then $\psi(Y)$ is distributed uniformly on $\psi(G)$
\end{lemma}
\begin{proof}
This follows from the observation that all cosets of $\ker(\psi)$ have equal size. 
\end{proof}

\autoref{lem:uniform_homomorphism} applied to the natural projection $\varphi_i \colon G \to G^{\Delta_i}$ implies a corollary of \autoref{lem:prob_of_moving}.
\begin{corollary}
\label{cor:movement_separation}
Let $G \leq S_n$  and $i \sim j$, then $\mathbb{E}(Y^{i \mapsto j}) = \Pr(Y(i)=j) = \frac{1}{|G.i|}$. In particular,
\[
\Pr(X(i)=j) =  \begin{cases} 
      (1-p)\frac{1}{|G.i|}+p\frac{1}{n} \geq \frac{1}{n} & i \sim j  \\
      p\frac{1}{n} & \text{otherwise}   \end{cases}.
      \]
\end{corollary}

An alternative choice for a random variable which  distinguishes $i \sim j$ from $i \not\sim j$ is 
\[
\mydef{X^{i \sim j}} = [X.i = X.j] = \begin{cases} 1 & i \text{ and } j \text{ belong to the same cycle of } X \\
0 & \text{otherwise} \end{cases}.
\]
For $G=S_n$, it is well-known that $\textrm{Pr}_{\sigma \sim U_G}(\sigma.i = \sigma.j) = \frac{1}{2}$, however, for other groups, this distribution is not so well-behaved. Nonetheless, it is guaranteed still that the expected values of $X^{i \sim j}$ seperate.
\begin{lemma}
Let $G \leq S_n$ and consider $X=X(G,p)$ for some $p \in [0,1]$. Then
\[
\textrm{Pr}(X^{i \sim j}) = \begin{cases} (1-p)\mathbb{E}(Y^{i \sim j}) +\frac{p}{2} & i \sim j \\
\frac{p}{2} & i \not\sim j \end{cases}
\]
\end{lemma}
We remark that $\mathbb{E}(Y^{i \sim j})$ can range from $1/n$ to $1$. The low end is witnessed by the worst-case-scenario of $G = \mathbb{Z}_2^m$, in which case $Y^{i \sim j} = Y^{i \mapsto j}$. The high end is witnessed by the cyclic group of order $n$. Thus, there is no strict bounding relation between $\mathbb{E}(Y^{i \sim j})$ and $\mathbb{E}(Z^{i \sim j})$. In what follows, we provide rigorous analysis of orbit recovery using the random variable $X^{i \mapsto j}$ which is possible due to its uniform behaviour. However, in practice, $X^{i \sim j}$ performs much better in distinguishing orbit relationships, in particular, since $n$ often does not appear in the denominator of the expected values, as it does in \autoref{cor:movement_separation}.

\begin{algorithm}[!htpb]
  \SetKwIF{If}{ElseIf}{Else}{if}{then}{elif}{else}{}%
  \DontPrintSemicolon
  \SetKwProg{Orbit Agreement}{Orbit Agreement}{}{}
  \LinesNotNumbered
  \KwIn{$(X,i,j,N)$ \\ $\bullet$ The ability to sample from the distribution of $X$ \quad $\bullet$ Any two points $i,j \in [n]$ \\ $\bullet$  $N$ a sample size}
  \KwOut{\texttt{true} if $i \sim j$ and \texttt{false} otherwise}
  \nl Let $c = (1+\tilde{p})/(2n)$ and $\delta = (1-\tilde{p})/2n$ for any $\tilde{p}>p$.\;
  \nl Determine if $\mathbb{E}(X^{i \mapsto j})>c$ using the mean-threshold distinguisher on $N$ samples. \;
  \nl \textbf{return} \texttt{true} if $\mathbb{E}(X^{i \mapsto j})>c$ and \texttt{false} otherwise. 
  \caption{Orbit Agreement\label{alg:OrbitAgreement}}
\end{algorithm}

\noindent The following lemma is a direct application of Hoeffding's inequality.
\begin{lemma}
\autoref{alg:OrbitAgreement} succeeds with probability at least $1-\alpha$ when $p < \widetilde{p}$ and $N> N\left(\alpha,\frac{1-\tilde{p}}{2n}\right)$. 
\end{lemma}

Rather than compounding error by iterating \autoref{alg:OrbitAgreement}, one may use the frequency vector of the random variable $X(i)$ to identify the orbit $G.i$, all at once.

\begin{theorem}
\label{thm:single_orbit_theorem} \autoref{alg:SingleOrbitRecovery} errs  with probability at most $\alpha$ when $p< \widetilde{p}$ and~$
N>~N\left(\frac{\alpha}{n},\frac{1-\widetilde{p}}{2n}\right).$
\end{theorem}
\begin{proof}
The probability that \autoref{alg:SingleOrbitRecovery} miscategorizes a point as belonging to an orbit is 
\[
\Pr(u_j>c \text{ for } i \not\sim j) \leq \sum_{i \not\sim j} \Pr(u_j>c)  \leq \sum_{i \not\sim j} e^{-2\delta^2N} = (n-|G.i|)e^{-2\delta^2N}
\]
for $\delta = c-\frac{\widetilde{p}}{n} = \frac{1-\widetilde{p}}{2n}$. 
The miscategorization probability of not returning a point in the orbit of $i$ is 
\[
\Pr(u_j < c \text{ for some } i\sim j) \leq \sum_{i \sim j} \Pr(u_i<c) \leq \sum_{i \sim j} e^{-2\delta^2N} = |G.i|e^{-2\delta^2N}
\]
so the probability that either error happens is bounded by their sum
\[
\Pr(|G.i| \text{ not identified correctly by }\autoref{alg:SingleOrbitRecovery}) \leq ne^{-2\delta^2N}
\]
which is bounded by $\alpha$ whenever $N \geq N\left(\frac{\alpha}{n},\delta\right)$. 
\end{proof}

\begin{algorithm}[!htpb]
  \SetKwIF{If}{ElseIf}{Else}{if}{then}{elif}{else}{}%
  \DontPrintSemicolon
  \SetKwProg{Single Orbit Recovery}{Single Orbit Recovery}{}{}
  \LinesNotNumbered
  \KwIn{$(X,i,N)$ \\ $\bullet$ The ability to sample from the distribution of $X$  \quad  $\bullet$ Any point $i \in [n]$ \\ $\bullet$  $N$ a sample size \\}
  \KwOut{The orbit of $G$ containing $i$}
  \nl Create the vector $u=(u_1,\ldots,u_n)$ where $u_j$ is the number of observations $X(i)=j$, of $N$. \;
  \nl Set $c = (1+\tilde{p})/(2n)$ and $\delta = (1-\tilde{p})/2n$ for any $\tilde{p}>p$.\;
  \nl \textbf{return} the collection of $j$ for which $u_j>c$.\;
  \caption{Single Orbit Recovery\label{alg:SingleOrbitRecovery}}
\end{algorithm}

\begin{remark}[The marginals $X(i)$ are not independent]
\label{rem:dependence} The variables $X(i)$ and $X(i')$ are not independent. For example, for transitive $G$ and $i \neq j$
\[
0 = \textrm{Pr}(X(i)=1 \text{ and } X(i')=1) \neq \textrm{Pr}(X(i)=1) \cdot \textrm{Pr}(X(i')=1) = \frac{1}{n^2}. 
\]
Thus, it is not statistically valid to use the same sample of observations $X_1,\ldots,X_N$  to perform single orbit recovery on multiple orbits whilst claiming the confidence levels of \autoref{thm:single_orbit_theorem}. 
\end{remark}

\begin{algorithm}[!htpb]
  \SetKwIF{If}{ElseIf}{Else}{if}{then}{elif}{else}{}%
  \DontPrintSemicolon
  \SetKwProg{Orbit Recovery}{Orbit Recovery}{}{}
  \LinesNotNumbered
  \KwIn{$(X,N)$ \\ $\bullet$ The ability to sample from the distribution of $X$ \\ $\bullet$  $N$ a number of samples}
  \KwOut{The orbits of $G$}
  \nl Set \texttt{unclassified} equal to $[n]$ and set \texttt{orbits} equal to $\{\}$.\;
  \nl \While{$|$\texttt{unclassified}$|>1$}{\nl Pick $i \in \texttt{unclassified}$  and compute $G.i$ via  \autoref{alg:SingleOrbitRecovery}\;
    \nl Include $G.i$ as an element of \texttt{orbits} and remove each $j \in G.i$ from \texttt{unclassified}\;}
  \nl \textbf{return} \texttt{orbits}
  \caption{Orbit Recovery\label{alg:OrbitRecovery}}
\end{algorithm}

\begin{theorem}
\label{thm:orbit_recovery_theorem}
\autoref{alg:OrbitRecovery} recovers the orbits $\Delta=(\Delta_1,\ldots,\Delta_m)$ of $G$, of sizes $(n_1,\ldots,n_m)$ with confidence level $1-\alpha$ using   
\[
N \geq mN\left(\frac{1-\sqrt[m]{1-\alpha}}{2n},\frac{1-\widetilde{p}}{2n}\right) \leq nN\left(\frac{1-\sqrt[n]{1-\alpha}}{2n},\frac{1-\widetilde{p}}{2n}\right)
\]
samples of $X$. 
\end{theorem}
\begin{proof}
\autoref{alg:OrbitRecovery} recovers orbits by iteratively calling \autoref{alg:SingleOrbitRecovery} on each of the $m$ orbits, which \emph{a priori}, could be the finest partition with $m=n$. 

To obtain confidence level $1-\alpha$ on a process which requires $m$ subprocesses to be successful with confidence level $1-\widehat{\alpha}$, then one must take $\widehat{\alpha} \leq 1-\sqrt[m]{1-\alpha}$. Applying \autoref{thm:single_orbit_theorem} to this confidence level, we see that it is sufficient to use $N_i = N\left(\frac{1-\sqrt[m]{1-\alpha}}{2\widehat{n_i}},\frac{1-\widetilde{p}}{2\widehat{n_i}}\right)$ samples when classifying the $i$-th orbit where $\widehat{n_i}$ is the number of points of $[n]$ yet to be classified. The worst-case classification order of the orbits is to classify the orbits in increasing orbit-size, say $|\Delta_1|=n_1 \leq |\Delta_2|=n_2 \leq \cdots \leq |\Delta_m|=n_m$ in which case $ \widehat{n_i} = n-\sum_{j=1}^{i-1}n_i$. Equipped with this notation, \autoref{alg:OrbitRecovery} uses $N_i$ samples at step $i$. Since 
\[N_i \leq N\left(\frac{1-\sqrt[m]{1-\alpha}}{2\widehat{n_i}},\frac{1-\widetilde{p}}{2\widehat{n_i}}\right) \leq N\left(\frac{1-\sqrt[n]{1-\alpha}}{2n},\frac{1-\widetilde{p}}{2n}\right)\]
the result follows. 
\end{proof}

If $\Delta_i$ is an orbit of $G$, then one may sample from the transitive constituent $G^{\Delta_i}$ using $X(G,p)$ as follows. First define $\mathcal P$ to be the property of membership in the Young subgroup $S_{\Delta_i} \times S_{[n]-\Delta_i}$. This property gives the sampling procedure $X_{\mathcal P}$ which remains uniform on $G$. Moreover, the map $\psi_i:G \to G^{\Delta_i}$ which takes a permutation to its action on $\Delta_i$ is well-defined on any observation of $X_{\mathcal P}$. The distribution $\psi_i(X_{\mathcal P})$ on $S_{\Delta_i}$ is uniform when restricted to $G^{\Delta_i}$ by \autoref{lem:uniform_homomorphism}. We call $\mydef{X^{\Delta_i}}=\psi_i(X_{\mathcal P})$ a \mydef{transitive constituent sampler}.  The following method uses the orbit refining test and the transitivity test on $X^{\Delta_i}$ to confirm an orbit $\Delta_i$ of $G$ in ``two directions''.

\begin{algorithm}[!htpb]
  \SetKwIF{If}{ElseIf}{Else}{if}{then}{elif}{else}{}
  \DontPrintSemicolon
  \SetKwProg{Orbit Confirmation}{Orbit Confirmation}{}{}
  \LinesNotNumbered
  \KwIn{$(X,\Delta_i,N_1,N_2)$ \\ $\bullet$ The ability to sample from the distribution of $X$ \\ $\bullet$ $\Delta_i \subseteq [n]$ \\ $\bullet$  $(N_1,N_2)$ a pair of sample sizes\\}
  \KwOut{\texttt{true} if $\Delta_i$ is an orbit of $G$ and \texttt{false} otherwise}
  \nl Run the orbit refining test (with $N_1$ samples) on $H=S_{\Delta_i}\times S_{[n]-\Delta_i}$  and $X$ to confirm that $\Delta_i$ is a union of orbits. If not, \textbf{return} \texttt{false}\;
  \nl Apply the transitivity test \autoref{alg:transitivitytest} to $X^{\Delta_i}$ to determine if $G$ acts transitively on $\Delta_i$. \;
  \nl \textbf{return} \texttt{true} if the transitivity test passes and \texttt{false} otherwise.
  \caption{Orbit Confirmation\label{alg:OrbitConfirmation}}
\end{algorithm}

\begin{theorem} For any $G \leq S_n$ and  $\Delta_i \subset [n]$,
\autoref{alg:OrbitConfirmation} determines if $\Delta_i$ is an orbit of $G$ with probability at least $1-\alpha$ when $p < \widetilde{p}$ using $N=N_1+N_2$ samples of $X$, where 
\[
 N_1 \geq N\left(1-\sqrt{1-\alpha},\frac{1-\widetilde{p}}{4}\right) \quad \text{ and } N_2 \geq  N\left(1-\sqrt{1-\alpha},\frac{1-\widetilde{p}}{2|\Delta_i|}\right).
\]
\end{theorem}
\begin{proof}
As in \autoref{thm:orbit_recovery_theorem}, this process has two steps, so to obtain confidence level $1-\alpha$ in their concatenation, we must have $1-\sqrt{1-\alpha}$ confidence in each constituent algorithm. The relevant bounds for the constituents are given in \autoref{thm:subgroup_test} and \autoref{thm:transitivity_test_thm}. We remark that the bound may be tightened by using $\mathcal P = \sigma \in S_{\Delta_i} \times S_{[n]-\Delta_i}$ sample error detection to reduce $\widetilde{p}$ for~$N_2$. 
\end{proof}

The transitive constituent samplers may be used in replacement of $X$ in any group recovery algorithm (e.g. \autoref{alg:main_algorithm}) to recover the corresponding transitive constituents of $G$. 
\begin{algorithm}[!htpb]
  \SetKwIF{If}{ElseIf}{Else}{if}{then}{elif}{else}{}%
  \DontPrintSemicolon
  \SetKwProg{Transitive Constituent Recovery}{Transitive Constituent Recovery}{}{}
  \LinesNotNumbered
  \KwIn{$(X,\Delta,\mathcal A)$ \\ $\bullet$ The ability to sample from the distribution of $X$ \\ $\bullet$  $\Delta_1,\ldots,\Delta_m$ the orbits of $G$ \\
  $\bullet$ $\mathcal A$ an algorithm which recovers a group $G$ from some $X(G,p)$}
  \KwOut{The transitive constituents of $G$}
  \nl \textbf{return} $\{\mathcal A(X^{\Delta_i})\}_{i=1}^m$\;
  \caption{Transitive Constituent Recovery \label{alg:TransitiveConstituentRecovery}}
\end{algorithm}

The following heuristic algorithm disregards the dependence mentioned in \autoref{rem:dependence} and attempts to compute the orbits of $G$ using the $n \times n$ frequency table, recording frequencies of $X^{i \sim j}$ for $X_1,\ldots,X_N \sim X$.  After constructing this table, it finds a set \texttt{distinct} of pairs $(i,j)$ which are confidently deemed to be in distinct orbits. This provides information for the permutation error-detector
\[
\mathcal P: i \text{ and }j\text{ belong to distinct cycles of }\sigma \text{ for all }(i,j) \in \texttt{distinct}.
\] Equipped with this detector, one may sample again $N$ elements $X_1',\ldots,X_N'$ from $X_{\mathcal P}$ with a significant error-reduction. Again, the algorithm computes the frequency table and deems certain pairs to be in distinct orbits of $G$. Those pairs which are not deemed to be distinct form the edges of a graph whose connected components are returned. 

\begin{algorithm}[!htpb]
  \SetKwIF{If}{ElseIf}{Else}{if}{then}{elif}{else}{}%
  \DontPrintSemicolon
  \SetKwProg{Heuristic  Orbit Recovery}{Heuristic  Orbit Recovery}{}{}
  \LinesNotNumbered
  \KwIn{$(X,N,t)$ \\ $\bullet$ The ability to sample from the distribution of $X$ \\ $\bullet$  $N$ a number of samples \\ $\bullet$  A threshold $t$ (default $t=N\frac{\widetilde{p}}{2}$ for $\widetilde{p}>p$)\\ $\bullet$ (Optional) \texttt{Adaptive Version} or \texttt{Non-adaptive Version}}
  \KwOut{The orbits of $G$}
  \nl Sample $X_1,\ldots,X_N \sim X$ and construct the frequency matrix $T=(t_{ij})_{1 \leq i, j \leq n}$ of $X^{i \sim j}$\;
  \nl Declare those $(i,j)$ with $t_{ij}<t$ to be in distinct orbits.\;
  \nl Sample $X'_1,\ldots,X'_N \sim X_{\mathcal P}$ with $\mathcal P:\sigma.i \neq \sigma.j$ for $(i,j)$ in distinct orbits.\;
  \nl Create the new frequency matrix $T'=(t'_{ij})_{1 \leq i, j \leq n}$ and compute the connected components $C=(C_1,\ldots,C_k)$ of the graph whose edges are those $(i,j)$ such that $t'_{ij}>p/2$\;
  \nl \If{\textit{Non-adaptive}}{\textbf{return} C}
   \Else{
  \nl \If{\autoref{alg:OrbitConfirmation} passes }{\textbf{return} $C$}
  \Else{ \nl \textbf{if} orbit confirmation failed due to orbit refinement \textbf{then} decrease $t$\;
  \nl \textbf{if} orbit confirmation failed due to transitivity test \textbf{then} increase $t$\;
  \nl Go to \texttt{3} 
  } }
  \caption{Heuristic Orbit Recovery \label{alg:HeuristicOrbitRecovery}}
\end{algorithm}

\begin{example}
Consider $\mathbb{Z}_2^2$ acting on $\{1,2,3,4\}$ as the symmetries of a rectangle and $D_8$ acting on $\{5,6,7,8\}$ as the symmetries of a square and define $G = \mathbb{Z}_2^2 \times D_8$ acting on $\{1,2,3,4\} \cup \{5,6,7,8\}$. We fix $p = 0.25 \leq \widetilde{p} = 0.25 $. With $N=100$  we obtain the frequency matrix 
\[
\begin{bmatrix}
100 & 37 & 31 & 33 & 12 & 12 & 10 & 10 \\
37 & 100 & 33 & 30 & {\color{red}{13}} & 12 & 10 & {\color{red}{13}} \\
31 & 33 & 100 & 35 & 10 & 12 & {\color{red}{14}} & {\color{red}{14}} \\
33 & 30 & 35 & 100 & 12 & {\color{red}{13}} & 9 & {\color{red}{16}} \\
12 & {\color{red}{13}} & 10 & 12 & 100 & 42 & 56 & 45 \\
12 & 12 & 12 & {\color{red}{13}} & 42 & 100 & 45 & 49 \\
10 & 10 & {\color{red}{14}} & 9 & 56 & 45 & 100 & 42 \\
10 & {\color{red}{13}} & {\color{red}{14}} & {\color{red}{16}} & 45 & 49 & 42 & 100
\end{bmatrix}
\]
 whose $(i,j)$-th entry is the number of observations of $X$ for which $i$ and $j$ share a cycle. Here, $t = 12.5$. Thus, \autoref{alg:HeuristicOrbitRecovery} recognizes all $(i,j)$ which do not share an orbit other than those colored in red in the above matrix. Resampling with the permutation property which requires the points in distinct orbits to be in distinct cycles reduces the empirical error rate to be nearly zero. Our new frequency table is 
 \[
 \begin{bmatrix}
100 & 27 & 31 & 22 & 0 & 0 & 0 & 0 \\
27 & 100 & 22 & 31 & 1 & 0 & 0 & 0 \\
31 & 22 & 100 & 28 & 0 & 0 & 0 & 0 \\
22 & 31 & 28 & 100 & 0 & 0 & 0 & 0 \\
0 & 1 & 0 & 0 & 100 & 37 & 54 & 41 \\
0 & 0 & 0 & 0 & 37 & 100 & 40 & 53 \\
0 & 0 & 0 & 0 & 54 & 40 & 100 & 38 \\
0 & 0 & 0 & 0 & 41 & 53 & 38 & 100
\end{bmatrix}
 \] 
 The algorithm finishes by declaring only those pairs  with frequencies above $12.5$ as sharing an orbit and returns the connected components of the associated graph, as shown in \autoref{fig:orbitgraph}.
\begin{figure}[!htpb]\raisebox{-68pt}{
\includegraphics[scale=0.18]{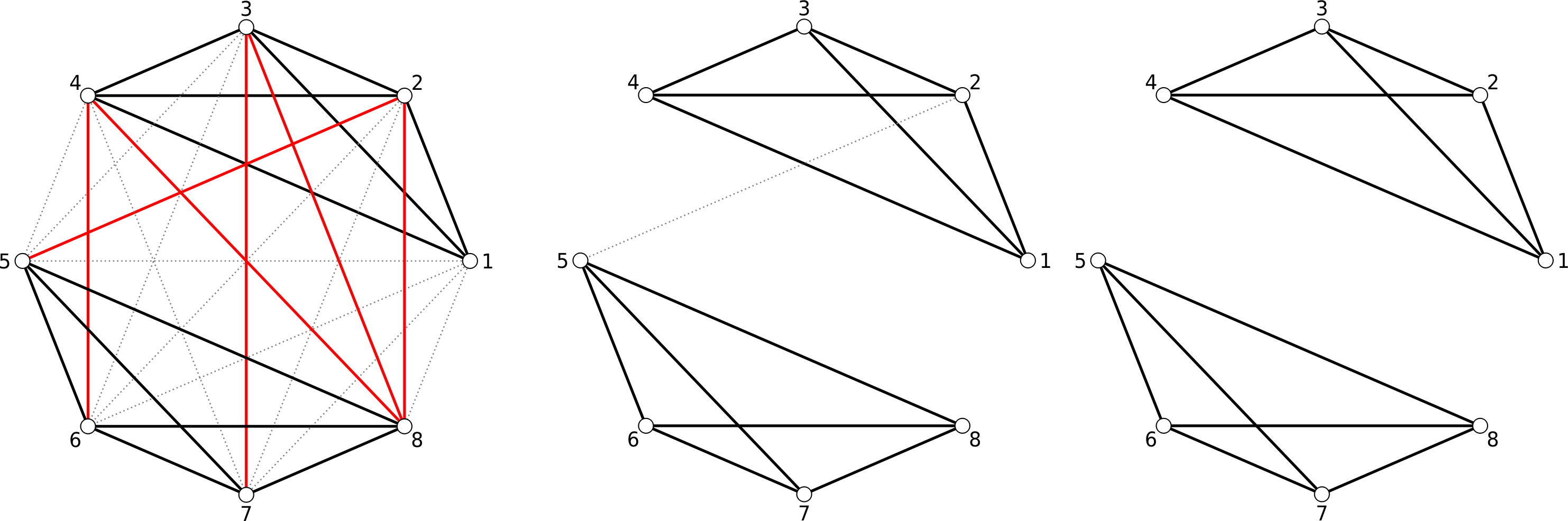}}
\caption{(Left) Gray edges represent pairs $(i,j)$ which rarely share a cycle in samples of $X$. Red edges are pairs which are in distinct orbits but are not recognized at the first step as such. (Center) The same graph on a new sample which performs error-detection based on the information indicated by gray edges in the left picture. (Right) The final output of \autoref{alg:HeuristicOrbitRecovery} are the connected components of this graph. }
\label{fig:orbitgraph}
\end{figure}
\end{example}

\subsection{Block recovery and primitivity test} 
Primitivity of a transitive permutation group $G$, and the minimal block structures of $G$, can be read-off from the orbits of the action $G^{(2)}$ of $G$ on ordered pairs, called \mydef{orbitals}. Given an orbital $\nabla$, the directed graph \mydef{\textrm{Graph}$(\nabla)$} on $[n]$ has an edge $i \mapsto j$ if and only if $(i,j) \in \nabla$. The orbital containing $(1,1)$ is the \mydef{diagonal orbital} whose graph is $n$ loops. 

\begin{lemma}[see Sections 3.2 and 3.6B of \cite{DixonPermutationGroups}]
\label{lem:orbitals} Let $G \leq S_n$ be transitive. If $\nabla$ is the orbital of $(i,j)$ for $i \neq j$ and $\textrm{Graph}(\nabla)$ is disconnected, then the vertices of the component containing $(i,j)$ form the smallest block of $G$ containing $i$ and $j$.
\end{lemma}

Thus, in order to compute the blocks of $G$, it suffices to compute the orbits of $G^{(2)}$ and appeal to \autoref{lem:orbitals}. In order to do this using $X$ and the orbit recovery algorithm \autoref{alg:OrbitRecovery}, one can interpret a sample of $X$ as acting on pairs of distinct elements of $[n]$, inducing a distribution we call $X^{(2)}$ on $S_{n(n-1)}$. We remark that $X^{(2)}$ does not have the form $(1-p)U_{G^{(2)}} + p U_{S_{n(n-1)}}$ since the image of $S_n$ acting on pairs of distinct elements is an isomorphic copy of $S_n$ properly contained in $S_{n(n-1)}$. Thus, there is some amount of uncertainty regarding whether our rigorous confidence bounds apply. We leave a more serious assessment of a block recovery test to future research. 

\begin{example}
\label{ex:primitive_not_two_transitive}
The cyclic group $C_5$ acting on $\{1,\ldots,5\}$ is primitive but not $2$-transitive. The graphs of the orbitals are given in the right of \autoref{fig:orbitals}. Since each is connected, we conclude that $C_5$ is primitive, however, $C_5$ is not $2$-transitive since there is more than one non-diagonal orbital.
Some orbital graphs of $D_6$ acting on $\{1,\ldots,6\}$ are not connected, as seen in the left of \autoref{fig:orbitals}. Thus $D_6$ is not primitive: it has block structures $\{\{1,3,5\},\{2,4,6\}\}$ and $\{\{1,4\},\{2,5\},\{3,6\}\}$. 
\begin{figure}[!htpb]
\begin{tikzpicture}[scale=0.6, every node/.style={circle, draw, fill=black, inner sep=2pt}, >={Stealth[scale=1.2]}]

  \begin{scope}[shift={(-3cm,0cm)}] 
\foreach \i in {1,...,5} {
    \node (G1\i) at (72*\i:2) {};
}
\draw[->] (G11) -- (G12);
\draw[->] (G12) -- (G13);
\draw[->] (G13) -- (G14);
\draw[->] (G14) -- (G15);
\draw[->] (G15) -- (G11);
\end{scope}

  \begin{scope}[shift={(2cm,0cm)}] 
\foreach \i in {1,...,5} {
    \node (G2\i) at (72*\i:2) {};
}
\draw[->] (G21) -- (G25);
\draw[->] (G25) -- (G24);
\draw[->] (G24) -- (G23);
\draw[->] (G23) -- (G22);
\draw[->] (G22) -- (G21);
\end{scope}

  \begin{scope}[shift={(7cm,0cm)}] 
\foreach \i in {1,...,5} {
    \node (G3\i) at (72*\i:2) {};
}
\draw[->] (G31) -- (G33);
\draw[->] (G33) -- (G35);
\draw[->] (G35) -- (G32);
\draw[->] (G32) -- (G34);
\draw[->] (G34) -- (G31);
\end{scope}

  \begin{scope}[shift={(12cm,0cm)}] 
\foreach \i in {1,...,5} {
    \node (G4\i) at (72*\i:2) {};
}
\draw[->] (G41) -- (G44);
\draw[->] (G44) -- (G42);
\draw[->] (G42) -- (G45);
\draw[->] (G45) -- (G43);
\draw[->] (G43) -- (G41);
\end{scope}

\begin{scope}[shift={(5cm,-6cm)}]
\foreach \i in {1,...,6} {
    \node (G5\i) at (60*\i:-2) {};
}
\draw[black] (G51) -- (G52) -- (G53) -- (G54) -- (G55) -- (G56) -- (G51);
\end{scope}

\begin{scope}[shift={(0cm,-6cm)}]
\foreach \i in {1,...,6} {
    \node (G6\i) at (60*\i:-2) {};
}
\draw[black] (G61) -- (G63);
\draw[black] (G63) -- (G65);
\draw[black] (G65) -- (G61);
\draw[black] (G62) -- (G64);
\draw[black] (G64) -- (G66);
\draw[black] (G66) -- (G62);
\end{scope}

\begin{scope}[shift={(10cm,-6cm)}]
\foreach \i in {1,...,6} {
    \node (G7\i) at (60*\i:-2) {};
}
\draw[black] (G71) -- (G74);
\draw[black] (G72) -- (G75);
\draw[black] (G73) -- (G76);
\end{scope}

\end{tikzpicture}
\caption{The (directed) graphs of orbitals of $C_5$ (top) and $D_6$ (bottom).}
\label{fig:orbitals}
\end{figure}
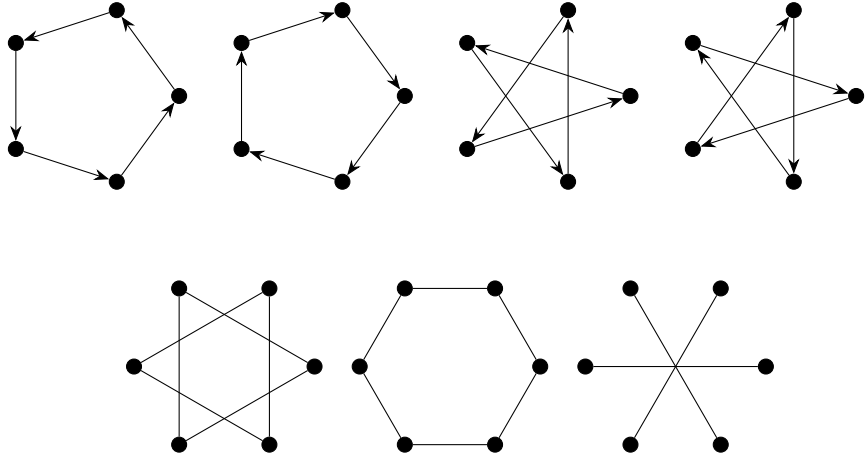
\end{example}
{\color{white}{.}}

From this construction, a cheap answer to the question of determining the block structures of $G$ (and in particular, primitivity of $G$) is to apply \autoref{alg:OrbitRecovery} to the random variable $X^{(2)}$ which interprets an observation of $X$ as acting on distinct ordered pairs, thus representing $G^{(2)}$.

\begin{algorithm}[!htpb]
  \SetKwIF{If}{ElseIf}{Else}{if}{then}{elif}{else}{}%
  \DontPrintSemicolon
  \SetKwProg{Minimal Block Recovery}{Minimal Block Recovery}{}{}
  \LinesNotNumbered
  \KwIn{$(X,N)$ \\ $\bullet$ The ability to sample from the distribution of $X=X(G,p)$ for transitive $G$  \\ $\bullet$  $N$ a number of samples}
  \KwOut{The minimal block structures of $G$}
  \nl Let $X'$ be the image of $X$ under $S_n \hookrightarrow S_{n\cdot (n-1)}$ defined via  action on distinct ordered pairs\;
  \nl Determine the orbits of $G^{(2)}$ (thus the orbitals of $G$) by applying \autoref{alg:OrbitRecovery} to $(X',N)$. \;
  \nl For each orbital $\nabla$, construct $\textrm{Graph}(\nabla)$ and compute its connected components $\{V_i\}_{i=1}^r$\;
  \nl \textbf{return} the vertices $\{V_i\}_{i=1}^r$ for each disconnected $\textrm{Graph}(\nabla)$ 
  \caption{Minimal Block Recovery \label{alg:BlockRecovery}}
\end{algorithm}
A primitivity test \autoref{alg:PrimitivityTest} appears as a byproduct in the trivial way.
\begin{algorithm}[!htpb]
  \SetKwIF{If}{ElseIf}{Else}{if}{then}{elif}{else}{}%
  \DontPrintSemicolon
  \SetKwProg{Primitivity Test }{Primitivity Test }{}{}
  \LinesNotNumbered
  \KwIn{$(X,N)$ \\ $\bullet$ The ability to sample from the distribution of $X=X(G,p)$ for transitive $G$ \\ $\bullet$  $N$ a number of samples}
  \KwOut{\texttt{true} if $G$ is primitive and \texttt{false} otherwise}
  \nl Run \autoref{alg:BlockRecovery}\;
  \nl \textbf{return} \texttt{false} if \autoref{alg:BlockRecovery} returns a block system and \texttt{true} otherwise. 
  \caption{Primitivity Test \label{alg:PrimitivityTest}}
\end{algorithm}

\begin{remark}
\label{rem:primitive_implied_by_two_transitive} The $2$-transitivity test \autoref{alg:transitivitytest} is a one-sided primitivity test since $2$-transitivity implies primitivity. As shown in \autoref{ex:primitive_not_two_transitive}, the converse fails. Another one-sided test approximates the proportion of imprimitive elements in $G$ via $X$. If few are observed, then $G$ is primitive, but the converse fails since there are primitive groups comprised entirely of imprimitive elements \cite{Arajuo}. 
\end{remark}

\section{Group recovery}
\label{sec:group_recovery}

We discuss how to combine our algorithms into our main method for recovering $G$ from $X(G,p)$. In particular, we discuss how our algorithms imply certain group error-detectors and permutation error detectors. We give some qualitative descriptors about how fast our hypothesis tests run, in terms of the number of samples to obtain high confidence. We also give an idea of how fast it is to evaluate the error-detectors we obtain with existing algorithms.

\subsection{Updating Bounds and error detectors:} In  \autoref{fig:flowchart}, the box ``update bounds and error detectors'' indicates that after each test, one should book-keep the following information earned:  any  implied group error detectors $\mathcal Q$ or  permutation error detectors $\mathcal P$ as well as any bounds on $p$,  $|G|$, or $B = {|S_{\mathcal P}|}/{n!}$.
Each hypothesis test in \autoref{sec:property_recovery} is presented as recovering a group-theoretic property $\mathcal Q$ about $G$, as given in \autoref{tab:Qtable}.

\begin{table}[!htpb]
\centering
\renewcommand{\arraystretch}{1}
{
\begin{tabular}{|l|| c| c| c| c|}
\hline
Algorithm / Output 
& $\mathcal Q$ 
& $|G|$ bound 
& Test Speed 
& $\mathcal Q(\widehat{G})$ speed \\
\hline\hline
\textbf{Giant Test} & & & & \\ \hline 
\quad true  & $\widehat G$ is a giant &  & Fast & Fast \\
\quad false & $\widehat G$ is a non-giant & $|G|\leq (n-1)!$ &  & Fast \\
\hline\hline
\textbf{Subgroup\_H Test} & & & & \\ \hline 
\quad true  & $\widehat G\leq H$ & $|G|\leq |H|$ & Fast & BSGS$(H)$ \\
\quad false & $\widehat G \nleq H$  & & &BSGS$(H)$  \\
\hline\hline
\textbf{$k$-Transitivity} & & & & \\ \hline 
\quad true $(k = 1)$ & $\widehat G$ is transitive & $|G|\leq 4^n$ & Slow &  Very Fast\\
\quad false $(k = 1)$  & $\widehat G$ is intransitive & &  & Very Fast\\
\quad true $(k > 1)$ & $\widehat G$ is $k$-transitive \& primitive & & Very Slow  & Fast \\
\quad false $(k > 1)$  & $\widehat G$ is not $k$-transitive & &  &  Fast  \\
\hline\hline
\textbf{Orbit Agreement} & & & & \\ \hline 
\quad  true $(i,j)$  & $G.i=G.j$ &  &  Slow & Very Fast   \\ 
\quad  false $(i,j)$  & $G.i\neq G.j$ & $|G| \leq (n-1)!$ &  Slow & Very Fast   \\ 
\hline\hline
\textbf{Orbit Recovery} & & & & \\ \hline 
\quad  $\Delta$ & $\widehat G\leq S_{\Delta}=H$ & $|G|\leq |H|$ &  Very Slow & Very Fast   \\ 
(orbits) & $\widehat G^{\Delta_i}$ is transitive & & & Very Fast\\
\hline\hline
\textbf{Block Recovery} & & & & \\ \hline 
\quad $\mathcal B_1,\ldots,\mathcal B_k$ & $G \leq \bigcap_{i=1}^k \textrm{Wr}(\mathcal B_i)=H$ &$ |G|\leq |H|$ & Very Slow & BSGS$(\widehat{G})$\\
 (minimal blocks) & $\widehat G \lneq \textrm{Wr}(\mathcal B')$ for any $\mathcal B'\leq \mathcal B_i$  & & & BSGS($\widehat G$)\\
\hline\hline
\textbf{Primitivity Test} & & & & \\ \hline 
\quad true  & $\widehat G$ is primitive & & Very Slow & BSGS($\widehat G$) \\
\quad false & $\widehat G$ is imprimitive &$|G| \leq 2(n/2)!^2$  & & BSGS($\widehat G$) \\
\hline
\end{tabular}}
\caption{Information earned from algorithms. Group properties and bounds on $|G|$.}
\label{tab:Qtable}
\end{table}

Often a permutation property $\mathcal P$ is obtained as a byproduct of our algorithms. This information for each relevant algorithm is given in \autoref{tab:Ptable} under the assumption that $G$ was already deemed a non-giant. In addition to this information \autoref{tab:Qtable} and \autoref{tab:Ptable} summarize the speed of each hypothesis test, the speed of evaluating the property discovered, and any bounds earned on $|G|$ or $B=|S_{\mathcal P}|/n!$.  We write $\textrm{BSGS}(H)$ for the cost of computing a \textit{base and strong generating set} of a group $H$ via Shreier-Sims \cite{Sims}. Note that because some properties are difficult to evaluate, an implementation of our main algorithm may choose to effectively ignore them.

\begin{table}
{
\begin{tabular}{|l|| c| c| c|}
\hline
Algorithm / Output 
& $\mathcal P$ 
& $B=|S_{\mathcal P}|/n!$ bound 
& $\mathcal P(\widehat{G})$ speed \\
\hline\hline
\textbf{Subgroup\_H Test} & & & \\ \hline 
\quad true  & $\sigma \in H$ & $|H|/n!$ & BSGS$(H)$ \\
\quad \quad ($H=A_n$) & $\sigma$ is even & $|H|/n! = \frac 1 2$ & Very Fast \\ 
\quad \quad ($H=S_\Delta$) & cycles of $\sigma$ refine $\Delta$ & ${{n}\choose{|\Delta_1|,\ldots,|\Delta_m|}}^{-1}$ & Very Fast \\ 
\quad\quad  ($H=\textrm{Wr}(B_1,\ldots, B_k)$) &  $\sigma$ respects $\mathcal B$ & $(n/k)!^kk!/n!$ & Very Fast \\ 
\quad false & $\sigma \not\in H$ & & BSGS$(H)$  \\
\hline\hline
\textbf{Orbit Agreement} & & &  \\ \hline 
\quad  false $(i,j)$  & $i,j$ are in distinct cycles of $\sigma$ & $\frac 1 2$ & Very Fast   \\ 
\hline\hline
\textbf{$k$-Transitivity} & & & \\ \hline 
\quad true $(k = 1)$ & $\sigma$ is not a giant & $1-\textbf{Giant}(n)/n!$ & Slow \\
\quad false $(k = 1)$  & $\sigma$ is not an $n$-cycle & $1-\frac{1}{n}$& Very Fast \\
\quad true $(k > 2)$ & $\sigma$ is not Jordan & $1-\textbf{Jordan}(n)/n!$ & Slow \\
\hline\hline
\textbf{Primitivity Test} & & & \\ \hline 
\quad true  & $\sigma$ is not Jordan  & $1-\textbf{Jordan}(n)/n!$ & Slow \\
\hline
\end{tabular}}
\caption{Information on permutation properties and bounds on $B$ for \textbf{non-giant} groups.
}
\label{tab:Ptable}
\end{table}

The only bound left to discuss is that of $\widetilde{p}$ as an upper bound for $p$. Several of our hypothesis tests relate $p$ with the expected value of the relevant random variable, measured via a sample mean $\overline{\mu}$. \autoref{tab:papproximations} lists these relationships and summarizes how error  of $\overline{\mu}$ propagates to $p$. We omit the $p$-bounds obtained from a true giant test: it is cumbersome to write them and ultimately unhelpful.

\begin{table}[!htpb]
{\footnotesize{
\begin{tabular}{|l|| l|  c|}
\hline
Algorithm / Output 
& Random Variable/Relation 
&  $p$-bounds from $\bar\mu=\mu\pm\epsilon$ \\
\hline\hline
\textbf{Giant} &\quad \quad \quad \quad $\mathcal G(X_1,X_2)$&   \\ \hline 
\quad true  
& $\begin{array}{l}
\mu \leq (1-p+p^2)u(n)+p(1-p)
\\
\mu \geq (1-p+p^2)\ell(n) 
\end{array}$
&  (omitted)
\\ \hline 
\quad false  
& $\begin{array}{l}
\mu\leq 2p(1-p)+p^2u(n)
\\
\mu \geq p^2\ell(n) 
\end{array}$ 
&$ p \leq \sqrt{\frac{\bar{\mu}+\epsilon}{\ell(n)}}$
\\ \hline  \hline 
\textbf{Subgroup\_H}&\quad \quad \quad \quad $X \in H$ &   \\ \hline 
\quad $H \neq S_n$,  true  
& $\mu = (1-p)+\frac{p|H|}{n!}$ 
& $\displaystyle
p \in \left(
\frac{1-(\bar\mu+\epsilon)}{1-\frac{|H|}{n!}},
\;\;
\frac{1-(\bar\mu-\epsilon)}{1-\frac{|H|}{n!}}
\right)$
\\ \hline 
\quad false  
&  $\mu \leq \frac{1-p}{2}+\frac{p|H|}{n!}$  
& $\displaystyle
p \le
\frac{\frac12-(\bar\mu-\epsilon)}{\frac12-\frac{|H|}{n!}}$
\\ \hline  \hline 
\textbf{Transitivity}&\quad \quad \quad \quad$\textrm{Fix}(X)$&   \\ \hline 
\quad   $m>1$
& $\mu=(1-p)m+p$ 
& $\displaystyle
p \in \left(
\frac{\bar\mu-\epsilon-m}{1-m},
\;\;
\frac{\bar\mu+\epsilon-m}{1-m}
\right)$
\\ \hline  \hline 
\textbf{Single Orbit}&\quad \quad \quad \quad$X(i)=j$ &   \\ \hline 
\quad $|G.i|<n$ and  $i \sim j$ 
& $\mu =\frac{1-p}{|G.i|}+\frac p n$ 
& $\displaystyle
p \in \left(
\frac{\bar\mu-\epsilon-\frac1{|G.i|}}{\frac1n-\frac1{|G.i|}},
\;\;
\frac{\bar\mu+\epsilon-\frac1{|G.i|}}{\frac1n-\frac1{|G.i|}}
\right)$
\\ \hline 
\quad $|G.i|<n$ and   $i \not\sim j$ 
& $\mu =\frac p n$ 
& $\displaystyle
p \in \left(
n(\bar\mu-\epsilon),
\;\;
n(\bar\mu+\epsilon)
\right)$
\\ \hline
&\quad \quad \quad \quad$X(i)\in G.i$ &   \\ \hline 
\quad $|G.i|<n$ 
& $\mu =1-\frac p n (n-|G.i|)$ 
& $\displaystyle
p \in \left(
\frac{n(1-\bar\mu)-n\epsilon}{n-|G.i|},
\;\;
\frac{n(1-\bar\mu)+n\epsilon}{n-|G.i|}
\right)$
\\ \hline
\end{tabular}}}
\caption{Bounds on $p$ obtained from confidence intervals on test statistics. }
\label{tab:papproximations}
\end{table}

\begin{example} 
Continuing our example of $G=W(E_6) \leq S_{27}$, we take a fixed sample $\Sigma$ of $N=100$  i.i.d. observations of $X=X(G,0.25)$ and compute sample means as in \autoref{tab:fixedsample}. We use a bound of $p \leq \widetilde{p} = \frac{1}{3}$.  We note that the ${{100}\choose{2}}$ observations of $\mathcal G(X_1,X_2)$ are not independent, and so the associated statistical analysis is only heuristic. The confidence intervals are conservatively computed via Hoeffding's inequality. Note that the distinguishing threshold of the test is within these intervals for $\textrm{Fix}(X)$ and $X(1)=2$ indicating that $N=100$ samples is not sufficient to draw reliable conclusions. Moreover, the bounds on $p$ obtained by applying \autoref{tab:papproximations} are seen to be unhelpful as they do not bound $p$ any lower than the assumed bound of $\widetilde{p} =\frac 1 3$

\begin{table}[!htpb]
{\footnotesize{
\begin{tabular}{|r|l|l|l|l|l|}
 \hline
& $\langle X_i, X_j \rangle$ a giant & $X \in A_{27}$ & $X \in \textrm{AGL}(3,3)$ & $\textrm{Fix}(X)$ & $X(1)=2$ \\ \hline \hline 
\# Observations & $4950=bin(100,2) $& $100$ & $100$ & $100$ &$100$  \\ \hline 
Sample Mean  & $0.5828 $& $0.86$  & $0.0$ & $1.15$ & $0.04$ \\ \hline 
Upper bound low mean & $0.5513$ & $0.5$ &$0.5$ & $1.00$ & $0.0370$ \\ \hline
Lower bound high mean & $0.9702$ & $0.8333$ &  $0.8333$& $1.666$ & $0.0379$ \\ \hline
Distinguishing threshold  & $0.7608$& $0.6666$ &$0.6666$  & $1.333$ & $0.0375$  \\ \hline 
99\% Confidence Interval  & $(0.5596,0.6059)$ & $(0.6972,1.0]$ & $[0.0,0.1627)$& $[0.0, 5.545)
$& $[0.0,0.2027)$ \\ \hline 
Conclusion  & Not a giant & $G \leq A_{27}$ & $G \neq \textrm{AGL}(3,3)$ & Transitive? & $i \not\sim 2$? \\ \hline 
Bounds on $p$  & $p \leq 0.798$ & $p \in (0.0,0.6056)$ & $-$  & $-$  & $-$   \\ \hline 
\end{tabular}}}
\caption{Computation of sample means of some random variables on a fixed sample of $100$ samples, $0.75$ of them uniform from $W(E_6)$ and $0.25$ uniform from $S_{27}$. }
\label{tab:fixedsample}
\end{table}

\end{example}

\subsection{Structure of main algorithm}

We give the complete structure of our main algorithm, allowing for freedom in several implementation choices including which hypothesis tests to prioritize.

\begin{algorithm}[!htpb]
  \SetKwIF{If}{ElseIf}{Else}{if}{then}{elif}{else}{}%
  \DontPrintSemicolon
  \SetKwProg{Giant Test}{Giant Test}{}{}
  \LinesNotNumbered
  \KwIn{\\ $\bullet$ The ability to sample from the distribution of $X = X(G,p)$ for $G \leq S_n$\\
  \textbf{Optional:} \\
  $\bullet$ An error detector property $\mathcal P$ for permutations in $G$ \\
  $\bullet$ An error detector property $\mathcal Q$ for the group $G$ \\
  $\bullet$ A bound $\tilde{p}>p$  \quad 
  $\bullet$ A bound $\widetilde{M}$ on $\lceil \log_2(|G|) \rceil$  \quad 
  $\bullet$ A bound $\widetilde{B}$ on $B = \frac{|S_{\mathcal P}|}{n!}$   }
  \KwOut{$G$}
  \nl Determine if $G$ is a giant using \autoref{alg:gianttest}. \;
  \If{$G$ is a giant}
   			{\nl Determine if $G=A_n$ via \autoref{alg:subgrouptest} applied to $H=A_n$ \;
   			\nl \textbf{if} $G=A_n$ \textbf{return} $A_n$ \textbf{ else } \textbf{return} $S_n$.\;}

   	\Else{ 
   		\nl Update Bounds and Error Detectors \;
   		\nl \If{Success Rate Check Passes for some $k$}
   		{\nl Run \textbf{NiAGRA} on $X$ and $k$ with $\mathcal P$ and $\mathcal Q$ as error detectors, and a sufficiently high $N$ so that the the mode $\widehat{G}$ is $G$ with the desired confidence level \; 
   		\nl \textbf{return} $\widehat{G}$ \; }
   		\Else{
   			\nl Obtain another property $\mathcal P$ or $\mathcal Q$ from any of the following tests:
   			\begin{itemize}
   			\item Transitivity Test \quad $\bullet$ Orbit Recovery\quad $\bullet$Transitive Constituent Recovery
   			\item Block Recovery \quad $\bullet$ Primitivity Test\quad $\bullet$ Find Supergroup
   			\end{itemize}
   			\nl \textbf{goto} \texttt{5}\;}
	    }
  
  \caption{Main Algorithm\label{alg:main_algorithm}}
\end{algorithm}

\section{Experiments}
\label{sec:experiments}
For many of our tests and algorithms we give  precise expected values of our test statistics based on group theoretic properties of $G$. Thus, experiments involving these algorithms would only be showcasing the failure of Hoeffding's inequality to be tight. One exception, however, is \autoref{alg:findsupergroup} which finds a supergroup of $G$. While we provide rigorous bounds on the probability that this algorithm succeeds, we give no insight into its utility; for example, what is the probability that a super group is returned other than $S_n$?

Another avenue for experimentation is the setting when a user has access to a fixed sample $\{X_1,\ldots,X_N\}$ of $\textrm{i.i.d.}$ observations of $X$, but does not have the ability to sample additional observations of $X$. In this case, our proven sample bounds which guarantee certain confidence may exceed $N$, though the intuition behind our algorithms may still offer insight. 

\subsection{Supergroups}
Finding supergroups of $G$ reliably from $X(G,p)$ is the task of \autoref{alg:findsupergroup}. As suggested in \autoref{ex:27linesSupergroups}, it is rare that either of the following occurs:
\begin{itemize}
\item[a)] the returned group fails to be a supergroup, 
\item[b)] the returned group is a non-giant supergroup other than $G$.
\end{itemize}
 (a) is rare because of our conservative Hoeffding test: even with levels as low as $1-\alpha = 80\%$, the mean-threshold distinguisher is using many more samples than necessary to obtain the requested confidence level $1-\alpha$. The fact that (b) is rare is more structural. For example, if $G$ is a maximal non-giant then (b) can never occur.  We provide experimental evidence for the rarity of these events based on whether $G$ is intransitive, imprimitive, or primitive.

There are $1593$ conjugacy classes of subgroups of $S_{10}$. For each, we take a representative $G$ and run \textit{Find Supergroup} $N=1000$ times with $p=\frac 1 3$. We check if $G \leq H$ symbolically to speed-up the experiment and mimic low error rates. We then count how many results are (a) giants, (b) $G$, or (c) a group between $G$ and $S_n$. We plot this triple of proportions in the $2$-simplex shown in \autoref{fig:triangledata} for intransitive groups, imprimitive (transitive) groups, and primitive groups. 
 
\begin{figure}[!htpb]
\includegraphics[scale=0.27]{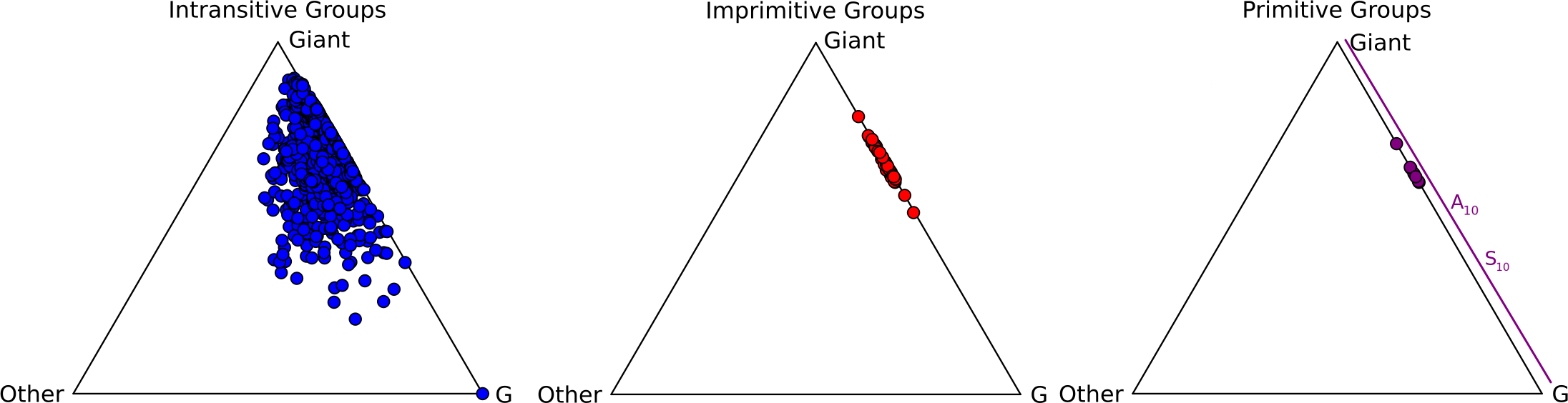}
\caption{The proportion of $N=1000$ runs of \textit{FindSupergroup} on each subgroup $G$ of $S_{10}$ which are (a) giants, (b) $G$ itself, or (c) a group in between. }
\label{fig:triangledata}
\end{figure} 
 
We provide some superlatives about our experiment below. 
\begin{itemize}
\item Among non-trivial intransitive groups, \textit{Find Supergroup} found $G$ most frequently for 
\begin{center}
 $G = C_7 \times S_1 \times S_1 \times S_1$
 \end{center}
\item Among non-giants, \textit{Find Supergroup} found a giant most frequently for
\[G=C_2 \times ((C_2  \times  C_2 \times  C_2)  \rtimes (C_2 \times  C_2))\] generated by 
\[
 g_1 = (1,9,6,3)(2,4,5,8) \quad g_2 =  (2,5)(3,9) \quad g_3 =
 (1,2)(3,8)(4,9)(5,6) \]
 \[ g_4 = 
 (1,6)(3,9)(7,10) \quad g_5 = 
 (1,5,6,2)(3,4,9,8)
\]
\item Among imprimitive groups,  \textit{Find Supergroup} found $G$ most frequently for  $G=C_{10}$.
\item Among imprimitive group,  \textit{Find Supergroup} found $G$ least often for  the  group $G$ of order $400$
generated by
\[
 g_1 = (1,6,4,9,7,10,8,3,5,2), \quad g_2 = 
 (2,9,6,3)(4,8,5,7), \quad g_3 = 
 (1,5,8,7,4)(2,6)(3,9)
\]
\item Among primitive groups,  \textit{Find Supergroup} found $G$ most frequently for $G=A_6 \rtimes C_2$. 
generated by 
\[
g_1 =   (1,6)(2,4)(3,10)(5,7)(8,9) \quad \quad g_2 =   (1,7,2)(3,8,5)(6,10,9)
\]
\item Among primitive groups,  \textit{Find Supergroup} found $G$ least often for the group $G$ of order $1440$ generated by 
\[
g_1 =  (2,7)(3,5)(9,10) \quad g_2 = 
 (1,9,10,5,4,6,8,3,7,2)
\]
\end{itemize}

The clear heuristic conclusion to draw from \autoref{fig:triangledata} is that for transitive non-giants, an effective algorithm for recovering $G$ is to amplify \textit{Find Supergroup} under the error-detector that $G$ is not a giant. The most frequent non-giant output for any such group is overwhelmingly the group itself. We leave it to future work to justify this statement rigorously. It is a phenomenon not fully explained by the proportion $\textbf{Giant}(10) \approx 0.254$ in \autoref{tab:jordan_giant_primitive}.

\subsection{Faster recovery via heuristic mean-thresholds} In this section, we test the performance of our algorithms on fixed sample sets.  In this setting, we sacrifice the rigorous confidence levels of our results, but we witness empirically that our ideas remain reliable.

\begin{figure}[!htpb]
\includegraphics[scale=0.4]{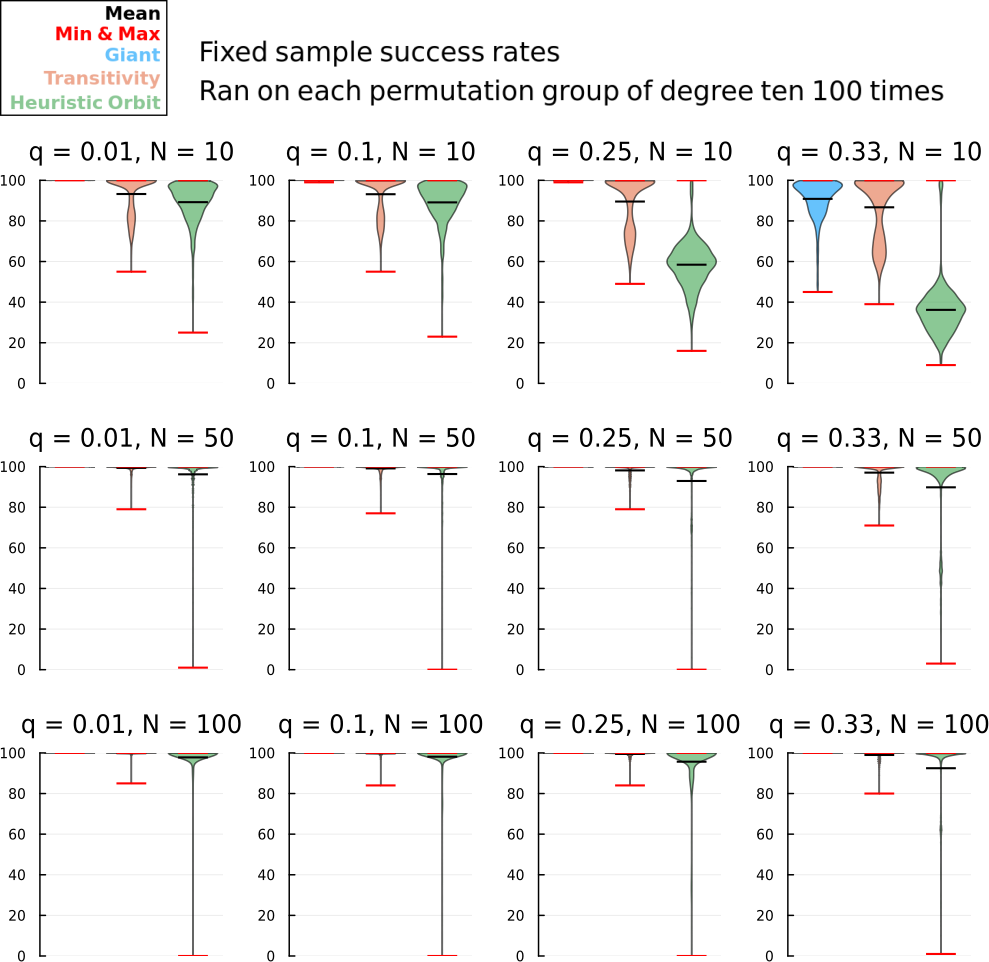}
\caption{For each pair $(q,N)$ shown, we display the violin plots of the success rates (of $100$ trials) of the giant test, transitivity test, and heuristic orbit recovery when applied to $100$ random samples of size $N$ pulled from the associated $X(G,p)$, as $G$ ranges through all $1593$ classes of subgroups of $S_{10}$.}
\label{fig:mle_experiment}
\end{figure}

We run the following experiment to assess how the giant test (\autoref{alg:gianttest}), transitivity test with $k=1$ (\autoref{alg:transitivitytest}), and \textit{non-adaptive} heuristic orbit recovery (\autoref{alg:HeuristicOrbitRecovery}) perform on fixed sample sizes. For each of the $1593$ conjugacy classes $[G]$ of subgroups of $S_{10}$ we pick $100$ random sample sets, each of some size $N$. Each sample is constructed by taking  $\lceil q\cdot N\rceil $ permutations uniformly and independently chosen from $S_{10}-G$ and the rest uniformly and independently chosen from $G$. On each of these $100$ samples, we run the giant test on the ${{N}\choose{2}}$ pairs of samples, the transitivity test on the $N$ samples, and the heuristic orbit recovery algorithm applied to the $N$ samples. For each class $[G]$, each $(q,N)$ pair, and each of these three algorithms, we obtain the number $\nu_{G,(q,N),\textrm{test}}$ of runs which were successful out of  $1593\cdot 3 \cdot 100 = 477900$ experiments. \autoref{fig:mle_experiment} shows violin plots for such $\nu_{G,(q,N),\textrm{test}}$ with $(q,N)$ and $``\textrm{test}"$ fixed. The $(q,N)$ range is $q \in \{0.01,0.1,0.25,0.33\}$ and $N \in \{10,50,100\}$, and $\tilde{p}$ is taken to be $0.33$. 

For all but the pair $(q,N) = (0.33,10)$, the giant test performs almost perfectly. While the transitivity test fails to be very reliable for $N=10$, it still succeeds more than half of the time and thus can be amplified, except for some groups when $(q,N)=(0.33,10)$. 

The heuristic orbit recovery results are the most interesting. For some groups, the performance is quite poor when $q=0.01$. This is due to a large difference between $q$ and the bound $\widetilde{p}$. For example, in the group $G = \mathbb{Z}_2^3 \times S_1 \times S_1$, the threshold in \autoref{alg:HeuristicOrbitRecovery} is $t = N \frac{\widetilde{p}}{2} = N\cdot 0.165$, which exceeds $N$ times the expected value of $Y^{i \sim j} = Y^{i \mapsto j}= \frac 1 8$. Thus, all points are decided, incorrectly, to be in their own orbit.  This error would be caught and handled in the \texttt{adaptive} version of \autoref{alg:HeuristicOrbitRecovery}. Our heuristic orbit recovery performs substantially better than our transitivity test, even though transitivity is a byproduct of orbit recovery. It would be interesting to try to improve any of our algorithms  beyond the proof-of-concept versions we have provided.

\bibliographystyle{abbrv}
\bibliography{myref}

\end{document}